\documentclass[twoside,11pt]{preprintCL}

\usepackage[T1]{fontenc}
\usepackage[english]{babel}
\usepackage{times}
\usepackage{amsmath}
\usepackage{amsfonts}
\usepackage{amssymb}
\usepackage{amsmath}
\usepackage{amsthm}
\usepackage{graphicx}
\usepackage{array}
\usepackage{color}
\usepackage{mathrsfs}
\usepackage{hyperref}
\usepackage{esint} 
\usepackage{tikz}
\usepackage{ upgreek }
\usepackage{enumitem}

\definecolor{darkgreen}{rgb}{0.1,0.7,0.1}

\newtheorem{theorem}{Theorem}
\newtheorem{lemma}{Lemma}[section]
\newtheorem{proposition}[lemma]{Proposition}

\newtheorem{remark}[lemma]{Remark}
\newtheorem{definition}[lemma]{Definition}

\makeatletter
\pgfdeclareshape{crosscircle}
{
  \inheritsavedanchors[from=circle] 
  \inheritanchorborder[from=circle]
  \inheritanchor[from=circle]{north}
  \inheritanchor[from=circle]{north west}
  \inheritanchor[from=circle]{north east}
  \inheritanchor[from=circle]{center}
  \inheritanchor[from=circle]{west}
  \inheritanchor[from=circle]{east}
  \inheritanchor[from=circle]{mid}
  \inheritanchor[from=circle]{mid west}
  \inheritanchor[from=circle]{mid east}
  \inheritanchor[from=circle]{base}
  \inheritanchor[from=circle]{base west}
  \inheritanchor[from=circle]{base east}
  \inheritanchor[from=circle]{south}
  \inheritanchor[from=circle]{south west}
  \inheritanchor[from=circle]{south east}
  \inheritbackgroundpath[from=circle]
  \foregroundpath{
    \centerpoint%
    \pgf@xc=\pgf@x%
    \pgf@yc=\pgf@y%
    \pgfutil@tempdima=\radius%
    \pgfmathsetlength{\pgf@xb}{\pgfkeysvalueof{/pgf/outer xsep}}%
    \pgfmathsetlength{\pgf@yb}{\pgfkeysvalueof{/pgf/outer ysep}}%
    \ifdim\pgf@xb<\pgf@yb%
      \advance\pgfutil@tempdima by-\pgf@yb%
    \else%
      \advance\pgfutil@tempdima by-\pgf@xb%
    \fi%
    \pgfpathmoveto{\pgfpointadd{\pgfqpoint{\pgf@xc}{\pgf@yc}}{\pgfqpoint{-0.707107\pgfutil@tempdima}{-0.707107\pgfutil@tempdima}}}
    \pgfpathlineto{\pgfpointadd{\pgfqpoint{\pgf@xc}{\pgf@yc}}{\pgfqpoint{0.707107\pgfutil@tempdima}{0.707107\pgfutil@tempdima}}}
    \pgfpathmoveto{\pgfpointadd{\pgfqpoint{\pgf@xc}{\pgf@yc}}{\pgfqpoint{-0.707107\pgfutil@tempdima}{0.707107\pgfutil@tempdima}}}
    \pgfpathlineto{\pgfpointadd{\pgfqpoint{\pgf@xc}{\pgf@yc}}{\pgfqpoint{0.707107\pgfutil@tempdima}{-0.707107\pgfutil@tempdima}}}
  }
}
\makeatother

\colorlet{symbols}{blue!90!black}
\colorlet{testcolor}{green!60!black}

\def\symbol#1{\textcolor{symbols}{#1}}
\def\1{\mathbf{\symbol{1}}}

\usetikzlibrary{shapes.misc}
\usetikzlibrary{shapes.symbols}
\usetikzlibrary{snakes}
\usetikzlibrary{decorations}
\usetikzlibrary{decorations.markings}

\def\drawx{\draw[-,solid] (-3pt,-3pt) -- (3pt,3pt);\draw[-,solid] (-3pt,3pt) -- (3pt,-3pt);}
\tikzset{
	root/.style={circle,fill=testcolor,inner sep=0pt, minimum size=2mm},
	dot/.style={circle,fill=black,inner sep=0pt, minimum size=1mm},
	var/.style={circle,fill=black!10,draw=black,inner sep=0pt, minimum size=2mm},
	varH/.style={rectangle,fill=black!10,draw=black,inner sep=0pt, minimum size=2mm},
	dotred/.style={circle,fill=black!50,inner sep=0pt, minimum size=2mm},
	generic/.style={semithick,shorten >=1pt,shorten <=1pt},
	dist/.style={ultra thick,draw=testcolor,shorten >=1pt,shorten <=1pt},
	testfcn/.style={ultra thick,testcolor,shorten >=1pt,shorten <=1pt,<-},
	testfcnx/.style={ultra thick,testcolor,shorten >=1pt,shorten <=1pt,<-,
		postaction={decorate,decoration={markings,mark=at position 0.6 with {\drawx}}}},
	kprime/.style={semithick,shorten >=1pt,shorten <=1pt,densely dashed,->},
	kprimex/.style={semithick,shorten >=1pt,shorten <=1pt,densely dashed,->,
		postaction={decorate,decoration={markings,mark=at position 0.4 with {\drawx}}}},
	kernel/.style={semithick,shorten >=1pt,shorten <=1pt,->},
	multx/.style={shorten >=1pt,shorten <=1pt,
		postaction={decorate,decoration={markings,mark=at position 0.5 with {\drawx}}}},
	kernelx/.style={semithick,shorten >=1pt,shorten <=1pt,->,
		postaction={decorate,decoration={markings,mark=at position 0.4 with {\drawx}}}},
	kernel1/.style={->,semithick,shorten >=1pt,shorten <=1pt,postaction={decorate,decoration={markings,mark=at position 0.45 with {\draw[-] (0,-0.1) -- (0,0.1);}}}},
	kernel2/.style={->,semithick,shorten >=1pt,shorten <=1pt,postaction={decorate,decoration={markings,mark=at position 0.45 with {\draw[-] (0.05,-0.1) -- (0.05,0.1);\draw[-] (-0.05,-0.1) -- (-0.05,0.1);}}}},
	kernelBig/.style={semithick,shorten >=1pt,shorten <=1pt,decorate, decoration={zigzag,amplitude=1.5pt,segment length = 3pt,pre length=2pt,post length=2pt}},
	rho/.style={dotted,semithick,shorten >=1pt,shorten <=1pt},
	renorm/.style={shape=circle,fill=white,inner sep=1pt},
	labl/.style={shape=rectangle,fill=white,inner sep=1pt},
	xi/.style={circle,fill=symbols!10,draw=symbols,inner sep=0pt,minimum size=1.2mm},
	xix/.style={crosscircle,fill=symbols!10,draw=symbols,inner sep=0pt,minimum size=1.2mm},
	xib/.style={circle,fill=symbols!10,draw=symbols,inner sep=0pt,minimum size=1.6mm},
	xibx/.style={crosscircle,fill=symbols!10,draw=symbols,inner sep=0pt,minimum size=1.6mm},
	not/.style={circle,fill=symbols,draw=symbols,inner sep=0pt,minimum size=0.5mm},
	>=stealth,
	}
\makeatletter
\def\DeclareSymbol#1#2#3{\expandafter\gdef\csname MH@symb@#1\endcsname{\tikz[baseline=#2,scale=0.15,draw=symbols]{#3}}\expandafter\gdef\csname MH@symb@#1s\endcsname{\scalebox{0.7}{\tikz[baseline=#2,scale=0.15,draw=symbols]{#3}}}}
\def\<#1>{\csname MH@symb@#1\endcsname}
\makeatother

\DeclareSymbol{Xi22}{0.5}{\draw (0,0) node[xi] {} -- (-1,1) node[not] {} -- (0,2) node[xi] {};}

\DeclareSymbol{Xi2}{-2}{\draw (0,-0.25) node[xi] {} -- (-1,1) node[xi] {};}
\DeclareSymbol{Xi3}{0}{\draw (0,0) node[xi] {} -- (-1,1) node[xi] {} -- (0,2) node[xi] {};}
\DeclareSymbol{Xi4}{2}{\draw (0,0) node[xi] {} -- (-1,1) node[xi] {} -- (0,2) node[xi] {} -- (-1,3) node[xi] {};}
\DeclareSymbol{Xi2X}{-2}{\draw (0,-0.25) node[xi] {} -- (-1,1) node[xix] {};}
\DeclareSymbol{XXi2}{-2}{\draw (0,-0.25) node[xix] {} -- (-1,1) node[xi] {};}

\DeclareSymbol{IXi2}{0}{\draw (0,-0.25) node[not] {} -- (-1,1) node[xi] {} -- (0,2) node[xi] {};}
\DeclareSymbol{IXi^2}{-1}{\draw (-1,1) node[xi] {} -- (0,0) node[not] {} -- (1,1) node[xi] {};}

\DeclareSymbol{XiX}{-2.8}{\node[xibx] {};}
\DeclareSymbol{Xi}{-2.8}{\node[xib] {};}
\DeclareSymbol{IXiX}{-1}{\draw (0,-0.25) node[not] {} -- (-1,1) node[xix] {};}

\DeclareSymbol{Xi3b}{-1}{\draw (-1,1) node[xi] {} -- (0,0) node[xi] {} -- (1,1) node[xi] {};}

\DeclareSymbol{IXi3}{2}{\draw (0,-0.25) node[not] {} -- (-1,1) node[xi] {} -- (0,2) node[xi] {} -- (-1,3) node[xi] {};}
\DeclareSymbol{IXi}{-2}{\draw (0,-0.25) node[not] {} -- (-1,1) node[xi] {};}
\DeclareSymbol{XiI}{-2}{\draw (0,-0.25) node[xi] {} -- (-1,1) node[not] {};}

\DeclareSymbol{Xi4b}{-1}{\draw(0,1.5) node[xi] {} -- (0,0); \draw (-1,1) node[xi] {} -- (0,0) node[xi] {} -- (1,1) node[xi] {};}
\DeclareSymbol{Xi4b'}{-1}{\draw(0,1.5) node[xi] {} -- (0,-0.2); \draw (-1,1) node[xi] {} -- (0,-0.2) node[not] {} -- (1,1) node[xi] {};}
\DeclareSymbol{Xi4c}{0}{\draw (0,1) -- (0.8,2.2) node[xi] {};\draw (0,-0.25) node[xi] {} -- (0,1) node[xi] {} -- (-0.8,2.2) node[xi] {};}
\DeclareSymbol{Xi4d}{-4.5}{\draw (0,-1.5) node[not] {} -- (0,0); \draw (-1,1) node[xi] {} -- (0,0) node[xi] {} -- (1,1) node[xi] {};}
\DeclareSymbol{Xi4e}{0}{\draw (0,2) node[xi] {} -- (-1,1) node[xi] {} -- (0,0) node[xi] {} -- (1,1) node[xi] {};}
\DeclareSymbol{Xi4e'}{0}{\draw (0,2) node[xi] {} -- (-1,1) node[xi] {} -- (0,-0.2) node[not] {} -- (1,1) node[xi] {};}

\newcommand{\bn}[1]{{[\kern-0.5ex] #1 
    [\kern-0.5ex]}}

\newcommand\symb[2][\bf]{{\mathchoice{\hbox{#1#2}}{\hbox{#1#2}}%
        {\hbox{\scriptsize#1#2}}{\hbox{\tiny#1#2}}}}
        
\def\Bess{\mbox{Bess}}

\def\Alg{\mbox{Alg}}

\def\dist{\mbox{dist}}

\def\R{{\symb R}}
\def\N{{\symb N}}
\def\Z{{\symb Z}}

\def\P{{\symb P}}

\def\Sp{\mbox{Sp}}

\def\un{\mathbf{1}}

\def\rm{\mathrm{m}}

\def\cKa{\cK^{(a)}}

\def\${|\!|\!|}

\def\Latt{|\!|\!|_{\Lambda}}

\renewcommand{\P}{\mathbb{P}}
\newcommand{\E}{\mathbb{E}}

\newcommand{\bbE}{\mathbb{E}}

\newcommand{\bbV}{\mathbb{V}}

\newcommand{\cA}{\mathcal{A}}
\newcommand{\cB}{\mathcal{B}}
\newcommand{\cC}{\mathcal{C}}
\newcommand{\cD}{\mathcal{D}}
\newcommand{\cE}{\mathcal{E}}
\newcommand{\cF}{\mathcal{F}}
\newcommand{\cG}{\mathcal{G}}
\newcommand{\cH}{\mathcal{H}}
\newcommand{\cI}{\mathcal{I}}
\newcommand{\cJ}{\mathcal{J}}
\newcommand{\cK}{\mathcal{K}}

\newcommand{\cM}{\mathcal{M}}

\newcommand{\cO}{\mathcal{O}}

\newcommand{\cQ}{\mathcal{Q}}
\newcommand{\cR}{\mathcal{R}}
\newcommand{\cS}{\mathcal{S}}
\newcommand{\cT}{\mathcal{T}}
\newcommand{\cU}{\mathcal{U}}

\newcommand{\cW}{\mathcal{W}}

\newcommand{\ccA}{\mathscr{A}}
\newcommand{\ccB}{\mathscr{B}}

\newcommand{\ccD}{\mathscr{D}}

\hypersetup{
citecolor=blue,
colorlinks=true, 
breaklinks=true, 
urlcolor= blue, 
linkcolor= black, 
bookmarksopen=true, 
pdftitle={Anderson Hamiltonian}, 
}

\begin{document}

\title{ {The continuous Anderson hamiltonian in $d\le 3$}}

\author{Cyril Labb\'e}
\institute{
Universit\'e Paris-Dauphine\footnote{PSL University, CNRS, UMR 7534, CEREMADE, 75016 Paris, France. \email{labbe@ceremade.dauphine.fr}}}

\vspace{2mm}

\date{\today}

\maketitle

\begin{abstract}
We construct the continuous Anderson hamiltonian on $(-L,L)^d$ driven by a white noise and endowed with either Dirichlet or periodic boundary conditions. Our construction holds in any dimension $d\le 3$ and relies on the theory of regularity structures~\cite{Hairer2014}: it yields a self-adjoint operator in $L^2\big((-L,L)^d\big)$ with pure point spectrum. In $d\ge 2$, a renormalisation of the operator by means of infinite constants is required to compensate for ill-defined products involving functionals of the white noise. We also obtain left tail estimates on the distributions of the eigenvalues: in particular, for $d=3$ these estimates show that the eigenvalues do not have exponential moments.

\medskip

\end{abstract}

\setcounter{tocdepth}{2}
\tableofcontents

\section{Introduction}

We are interested in the so-called Anderson hamiltonian in dimension $d\in \{1,2,3\}$ formally defined as
\begin{equation}\label{Eq:Hamiltonian}
\cH = -\Delta + \xi\;,\quad x\in (-L,L)^d\;,
\end{equation}
where $L>0$, $\xi$ is a white noise on $(-L,L)^d$ and $\Delta$ is the continuous Laplacian. The boundary conditions will be taken to be either Dirichlet or periodic.

\medskip

The operator $\cH$ belongs to the class of random Schr\"odinger operators that model, for instance, the evolution of a particle in a random potential $\xi$. An important question about such operators is whether there exist parts of the spectrum where the eigenfunctions are localised in space. This question was originally addressed in the discrete setting where the underlying space is $\Z^d$, starting with the work of Anderson~\cite{Anderson58} who showed that in dimension $3$ the lower part of the spectrum gives rise to localised eigenfunctions. This and many other related questions have given rise to a vast literature in physics and mathematics.

\medskip

The operator $\cH$ that we consider here arises - at least formally - as the continuum limit of discrete Anderson hamiltonians whose random potentials are given by i.i.d.~r.v.~with finite variance. As white noise is a random Schwartz distribution with H\"older regularity index slightly below $-d/2$, we are in a situation where the potential is singular and the construction of $\cH$ as a self-adjoint operator on $L^2((-L,L)^d)$ is far from being a simple task. While the case of dimension $1$ was covered by ``standard'' functional analysis arguments in a work of Fukushima and Nakao~\cite{Fukushima}, a renormalisation of the operator by means of infinite constants is required in higher dimensions. The case of dimension $2$ under periodic boundary conditions was carried out only recently by Allez and Chouk~\cite{AllezChouk} and relied on novel techniques coming from the field of stochastic analysis. One of the main achievements of the present paper is the construction of the operator in dimension $3$ under periodic and Dirichlet b.c., and in dimension $2$ under Dirichlet b.c. Note however that for the sake of completeness we present a construction that works in all dimensions $d\in\{1,2,3\}$ and for the two aforementioned types of boundary conditions.

\medskip

Localisation of the eigenfunctions at the bottom of the spectrum of $\cH$ in dimension $1$ and under Dirichlet b.c.~was addressed in~\cite{DL17}. It is shown therein that when the size $L$ of the segment goes to infinity, the smallest eigenvalues, rescaled as $(\log L)^{1/3}(\lambda_n +(\log L)^{2/3})$, converge to a Poisson point process of intensity $e^{x}dx$ and the corresponding eigenfunctions are localised: the localisation centers are asymptotically i.i.d.~uniform over $(-L,L)$, and each eigenfunction, considered on a space scale of order $(\log L)^{-1/3}$ around its localisation center, converges to the inverse of a hyperbolic cosine. Localisation of the eigenfunctions in higher parts of the spectrum is under investigation.

\medskip

A similar localisation phenomenon is expected at the bottom of the spectrum of $\cH$ in dimensions $2$ and $3$: we intend to address this question in future works, based on the present construction. Let us mention a work in progress by Chouk and van Zuijlen~\cite{CvZ} that determines the speed at which the lowest eigenvalue of $\cH$ in dimension $2$ goes to $-\infty$ when $L\to\infty$.\\

Let us now present briefly the reasons why the construction of $\cH$ is non-trivial. First of all, let us observe that while $\cH f$ is well-defined whenever $f$ is a smooth (say $\cC^2$) function, it never belongs to $L^2$. Indeed, the product $\xi\cdot f$ is not a function but only a distribution so that for $\cH f$ to belong to $L^2$ one needs a subtle cancellation to happen between $-\Delta f$ and $\xi\cdot f$, and this requires $-\Delta f$ itself not to be a function. Consequently the domain of $\cH$ does not contain any smooth functions, and therefore, the operator cannot be defined as the closure of its action on smooth functions.

In dimension $1$ Fukushima and Nakao~\cite{Fukushima} constructed the operator $\cH$ under Dirichlet b.c. Let us recall the main steps of their construction (note that it can be adapted to cover other types of boundary conditions). First one proves that the bilinear form
$$ \cE(f,g) = \int \nabla f \nabla g + \int \xi fg\;,$$
is closed in $H^1_0$. Then, a classical representation theorem allows to construct the resolvents. Due to the compactness of the injection of $H^1_0$ into $L^2$, the resolvents are compact, self-adjoint operators on $L^2$ so that they are associated with a self-adjoint operator $\cH$ with pure point spectrum. Note that the construction applies to any potential $\xi$ which is the distributional derivative of an almost surely bounded function. Let us also point out that the construction does not yield much information on the domain of $\cH$. However, one can guess that any element $f$ in the domain of $\cH$ should locally behave like $(-\Delta)^{-1} \xi$ so that the domain is made of random H\"older $3/2^-$ functions.

\medskip

In dimension $2$ and above, the term $\int \xi fg$ is no longer well-defined for $f,g \in H^1_0$, and it is possible to check that the bilinear form $\cE$ is not closable. In fact, the domain of the form itself is random: one needs to consider the sum $\nabla f \nabla g +\xi fg$ as a whole and hope for a cancellation to happen for its integral to make sense.\\
Actually, an additional difficulty appears. Since any element in the domain of $\cH$ should behave locally like $(-\Delta)^{-1} \xi$, the product $\xi \cdot (-\Delta)^{-1} \xi$ arises when applying the operator $\cH$ to any element in its domain: while this distribution is well-defined in dimension $1$ by Young's integration (recall that $(-\Delta)^{-1}$ improves regularity index by $2$), it falls out of the scope of deterministic integration theories as soon as $d\ge 2$. This term needs to be \textit{renormalised} by subtracting some infinite constant. More precisely, if one considers some regularised noise $\xi_\epsilon$, then it is possible to identify some diverging (in $\epsilon$) constant $c_\epsilon$ such that $\xi_\epsilon \cdot (-\Delta)^{-1} \xi_\epsilon - c_\epsilon$ converges in probability to a well-defined object as $\epsilon\downarrow 0$. Note that in dimension $3$, there are other ill-defined products that need to be renormalised.

\medskip

This suggests the following procedure. Given a regularised potential $\xi_\epsilon$, the corresponding operator $-\Delta + \xi_\epsilon$ is well-defined and its domain is $H^2$ (up to the choice of boundary conditions). From the above discussion, this sequence of operators does not converge as $\epsilon\downarrow 0$. Instead, one considers a \textit{renormalised} operator obtained by setting
$$ \cH_\epsilon := -\Delta + \xi_\epsilon + C_\epsilon\;.$$
for some appropriately chosen $C_\epsilon$. One then expects $\cH_\epsilon$ to converge, in some sense, to a limit that we call $\cH$. 

\medskip

Such a result was proven by Allez and Chouk~\cite{AllezChouk} in dimension $2$ and under periodic boundary conditions. To give a meaning to the limiting operator, they adopted the theory of paracontrolled distributions of Gubinelli, Imkeller and Perkowski~\cite{Max}. Let us point out that the theory of paracontrolled distributions and the theory of regularity structures, introduced by Hairer~\cite{Hairer2014}, are two independent approaches for solving singular stochastic PDEs such as the parabolic Anderson model or the stochastic quantization equation. Other related theories have been proposed since then, see~\cite{BBF,OW}. In the present paper, we rely on the theory of regularity structures to construct the limiting operator $\cH$. Let us observe that in dimension $d\ge 4$, none of these theories apply anymore.\\
Near the completion of the present article, we were informed of a very recent work~\cite{Max2}, based on paracontrolled distributions, that constructs the operator in dimensions $2$ and $3$ under periodic boundary conditions, obtains several interesting functional inequalities and solves semi-linear PDEs involving this hamiltonian.

\medskip

Our main result is the following. Let $\rho$ be an even, smooth function integrating to $1$ and supported in the unit ball of $\R^d$. Set $\rho_\epsilon(\cdot) := \epsilon^{-d}\rho(\cdot / \epsilon)$ for any $\epsilon > 0$, and consider the noise $\xi_\epsilon$ obtained by convolving white noise $\xi$ with $\rho_\epsilon$.

\begin{theorem}\label{Th:Main}
In any dimension $d\in\{1,2,3\}$ and under periodic or Dirichlet b.c., there exists a self-adjoint operator $\cH$ on $L^2((-L,L)^d)$ with pure point spectrum such that the following holds. For some suitably chosen sequence of constants $C_\epsilon$, as $\epsilon \downarrow 0$ the eigenvalues/eigenfunctions $(\lambda_{\epsilon,n},\varphi_{\epsilon,n})_{n\ge 1}$ of $\cH_\epsilon$ converge in probability to the eigenvalues/eigenfunctions $(\lambda_{n},\varphi_{n})_{n\ge 1}$ of $\cH$.
\end{theorem}
%
\begin{remark}\label{Remark:CV}
As it is stated, the convergence of the eigenfunctions is ambiguous since the limiting eigenvalues may have multiplicity larger than $1$ and then the corresponding eigenfunctions are not canonically defined. The convergence should be understood in the following way: if $\Lambda_1 < \Lambda_2 < \ldots$ are the successive, \textit{distinct} eigenvalues of $\cH$ and if $m_i \in \N$ are their multiplicities, then for any $i\ge 1$ the unit ball of the linear span of $\varphi_{\epsilon,1}, \ldots, \varphi_{\epsilon,M_{i}}$ converges for the Hausdorff metric to the unit ball of the linear span of $\varphi_{1}, \ldots, \varphi_{M_{i}}$, where $M_i := \sum_{j=1}^i m_j$.
\end{remark}
\begin{remark}
In dimension $1$, $C_\epsilon$ can be taken equal to $0$. In dimensions $2$ and $3$, $C_\epsilon$ diverges at speed $\log \epsilon$ and $1/\epsilon$ respectively: we refer to Subsection \ref{Sec:Models} for the precise expressions of these quantities. In any cases, these constants do not depend on the given boundary conditions.
\end{remark}

Our second main result is an estimate on the left tail of the distributions of the eigenvalues.

\begin{theorem}\label{Th:Tail}
In the context of the previous theorem, for any $n\ge 1$ there exist two constants $a > b >0$ such that for all $x>0$ large enough we have
\begin{equation}\label{Eq:Asympt}
e^{-a x^{2-d/2}} \le \P(\lambda_n < -x) \le e^{-b x^{2-d/2}}\;.
\end{equation}
\end{theorem}

In dimension $1$ and under periodic b.c., a more precise result was established by Cambronero, Rider and Ram\'irez~\cite{CamRidRam} on the first eigenvalue. In dimension $2$ and under periodic b.c., this result was established by Allez and Chouk~\cite{AllezChouk} for the first eigenvalue - they also conjectured the present result in dimension $3$. As a corollary of Theorem \ref{Th:Tail}, we deduce that the solution to the parabolic Anderson model constructed in~\cite{Hairer2014,HaiPar}
$$\partial_t u = \Delta u + u\cdot\xi\;,\quad x\in (-L,L)^d\;,$$
has no moments in dimension $3$, and has finite moments in dimension $2$ up to some finite time.

\bigskip

We now present the main steps of the proofs. To prove Theorem \ref{Th:Main}, we construct the resolvent operators through a fixed point problem. Then we show that they are continuous (in some sense) w.r.t.~the driving noise and that they are compact and self-adjoint operators. Once this is established, the result follows from classical arguments.\\
The resolvent $G_a = (\cH + a)^{-1}$ applied to some function $g\in L^2$ should be the fixed point of the map
\begin{equation}\label{Eq:fixedpt} f \mapsto (-\Delta + a)^{-1} g - (-\Delta + a)^{-1} (f \cdot \xi)\;.\end{equation}
To deal with ill-defined products, we lift this fixed point problem into an appropriate regularity structure (the same as the one required to solve the parabolic Anderson model, see~\cite{Hairer2014,mSHE}). Since $g\in L^2$, the natural setting for solving \eqref{Eq:fixedpt} is an $L^2$-type space\footnote{Actually, in dimension $3$ for technical reasons we need to work in $L^p$ with $p$ slightly larger than $2$.}: indeed, solving the equation in an $L^\infty$-type space would require to embed $L^2$ into such a space and this would induce a too large loss of regularity. Therefore, we rely on Besov-type spaces of modelled distributions introduced in~\cite{Recons}.

Note that we deal with an elliptic PDE, while the theory of regularity structures was applied so far to parabolic PDEs. To obtain contractivity of the fixed point map, we cannot play with the length of the time-interval as in the parabolic setting: instead, we play with the parameter $a>0$ which induces some exponential decay on the Green's function of the operator $-\Delta +a$. To implement this idea into regularity structures, one needs to consider a model $(\Pi^{(a)},\Gamma^{(a)})$ that depends on the value $a$ and whose associated integration kernel is given by the Green's function of $-\Delta +a$. In this context, it is more natural to introduce a cutoff on the Green's function at scale $1/\sqrt a$ around the origin rather than at scale $1$: this induces some minor differences in the definitions of the norms and the corresponding analytical bounds.

\smallskip

Let us explain the main difficulties coming from the boundary conditions. Most of the PDEs solved with the theory of regularity structures have been taken under periodic boundary conditions: this choice of b.c. does not induce any specific difficulty in our setting. On the other hand, if one opts for Dirichlet b.c.~then the corresponding Green's function is no longer translation invariant so that the construction of the model and the identification of the renormalisation constants may become involved.

In a recent work~\cite{Mate} on parabolic SPDEs with b.c., Gerencs\'er and Hairer presented a nice trick to circumvent this difficulty. Using the reflection principle, one can write the Green's function under Dirichlet b.c.~as the sum of the Green's function on the whole space and a series of shifted versions of this same function. The singularities of these shifted Green's functions are localised at the boundary of the domain. Thus, one builds the model $(\Pi^{(a)},\Gamma^{(a)})$ with the (translation invariant!) Green's function of $-\Delta + a$ on the whole space, and one deals ``by hand'' with the remaining kernels. This last part involves adding some weights near the boundary in the spaces of modelled distributions.\\
An important difficulty in our case comes from the interplay of these weights with the $L^2$ setting. Indeed, to obtain a genuine distribution out of a modelled distribution with weights near the boundary, one needs these weights to be ``summable'': in $L^2$ the summability condition becomes tremendously worse than in $L^\infty$ (see Theorem \ref{Th:Recons}), and a quick look at the regularity of the objects at stake suggests that one cannot hope for closing the map \eqref{Eq:fixedpt} even in dimension $2$.\\
This is where we use the fact that the Green's function vanishes at the boundary: a careful analysis of the associated convolution operator (see Theorem \ref{Th:ConvolAbs}) then allows to get some decay of the solution near the boundary and provides the required summability of the weights.\\
The present techniques are probably not sufficient to cover the case of Neumann b.c., but let us mention the work in progress~\cite{PAMMate} that should provide the required arguments to deal with this case.

\medskip

The proof of Theorem \ref{Th:Tail} essentially follows the strategy presented in~\cite{AllezChouk}. For simplicity, we consider the operator on $(-1,1)^d$ since the size of the domain does not play any role here and since $L$ will be used as a scaling parameter.\\
The key observation is the following: the $n$-th eigenvalue $\lambda_n$ of the Anderson hamiltonian on $(-1,1)^d$ coincides (up to a correction term due to renormalisation) with $L^2\, \tilde{\lambda}_n$ where $\tilde{\lambda}_n$ is the $n$-th eigenvalue of the Anderson hamiltonian on $(-L,L)^d$ with potential $L^{d/2-2} \xi$, see Lemma \ref{Lemma:IdLaw}. Taking $x \asymp L^2$, we deduce that to obtain \eqref{Eq:Asympt} it suffices to bound from above and below the probability that $\tilde{\lambda}_n < -c$ for some constant $c>0$, uniformly over all large $L$.\\
The latter eigenvalue should be very close to $0$ with large probability since the noise term vanishes as $L$ goes to infinity: hence $\tilde{\lambda}_n < -c$ should be a large deviation event. In particular, to prove the lower bound, we build some deterministic potential $h$ whose $n$-th eigenvalue is less than $-c$ and then use the Cameron-Martin Theorem to estimate the probability that $\tilde{\xi}$ is close to $h$: in our context, this part requires to adapt some arguments from~\cite{HaiPar, HairerQuastel} on generalised convolutions encoded by labelled graphs, see the end of Subsection \ref{Subsec:Techos}.


\subsection*{Organisation of the article}
All the intermediate results are presented under Dirichlet b.c.~since this is the most involved setting. However, each time a definition, statement or proof needs to be substantially modified to deal with periodic b.c.~, we make a remark starting by (Periodic).\\
Section \ref{Sec:Abstract} introduces the material from the theory of regularity structures that we will need, and presents some slight modifications with the original setting together with a technical bound on the norm of admissible models. Section \ref{Sec:Convol} constructs the abstract convolution operator with the Green's function of $-\Delta + a$ and establishes a fixed point result that allows us in Section \ref{Sec:Anderson} to construct the resolvents of $\cH$. The remainder of Section \ref{Sec:Anderson} is devoted to the proofs of Theorems \ref{Th:Main} and \ref{Th:Tail}.

\subsection*{Notations}

We let $|x|$ denote the Euclidean norm of $x\in\R^d$ and we let $B(x,r)$ denote the ball of radius $r>0$ centred at $x$. With a slight abuse of notation, we set $|k| = \sum_{i=1}^d k_i$ for any $k\in\N^d$. We let $\cC^r$ be the set of all maps from $\R^d$ into $\R$ with a bounded $k$-th derivative for any $k\in\N^d$ such that $|k| \le r$. We further define $\ccB^r$ to be the set of all functions $f$ supported in $B(0,1)$ and whose $\cC^r$-norm is bounded by $1$. For any function $\varphi$ on $\R^d$, any $x\in\R^d$ and any $\lambda > 0$ we set
$$ \varphi^\lambda_x(\cdot) := \lambda^{-d} \varphi\big((\cdot-x)/\lambda\big)\;.$$
From now on, $L$ will always be larger than or equal to $1$.

\subsection*{Acknowledgements} I would like to express my gratitude to Khalil Chouk for the enlightening discussions that we had on this topic. I acknowledge the project SINGULAR ANR-16-CE40-0020-01 for its financial support.

\section{The abstract setting}\label{Sec:Abstract}

\subsection{Kernels}
\begin{definition}\label{Def:Kernel}
Fix $r\in \N$ and $n_a \ge 0$. We say that $P:\R^d \to \R$ is an admissible kernel if there exists a collection of smooth functions $P_-$ and $P_n:\R^d\to\R$, $n\ge n_a$ such that:\begin{enumerate}
\item $P = P_- + P_+$ where $P_+:=\sum_{n\ge n_a} P_n$,
\item for every $n\ge n_a$, $P_n$ is supported in the set $\{x\in\R^d: |x| \le 2^{-n}\}$,
\item for every given $k\in\N^d$, there exists $C>0$ such that for all $n\ge n_a$ and all $x\in\R^d$ we have
$$ | \partial^k P_n(x)| \le C2^{n(d-2+|k|)}\;.$$
\item For every $n\ge n_a$, $P_n$ annihilates all polynomials of degree at most $r$.
\end{enumerate}
\end{definition}

This is the same setting as in~\cite[Section 5]{Hairer2014} except that we restrict here to the translation invariant case where $P(x,y) = P(x-y)$ and that our cutoff on the kernel occurs at scale $2^{-n_a}$ instead of $1$: this is because we will eventually consider the Green's function of the operator $-\Delta +a$ which has exponential decay outside the ball of radius $2^{-n_a}$ where $n_a$ is the smallest integer such that $2^{-n_a} \le 1/\sqrt a$.

\subsection{Regularity structure and models}\label{Subsection:RS}

A regularity structure is a triplet $(\cA,\cT,\cG)$ satisfying the following properties. The \textit{set of homogeneities} $\cA$ is a subset of $\R$, locally finite and bounded below. The \textit{model space} $\cT$ is a graded vector space $\oplus_{\beta \in \cA} \cT_\beta$ where each $\cT_\beta$ is a Banach space. The \textit{structure group} is a group of continuous linear transformations on $\cT$ such that for every $\Gamma\in\cG$, every $\beta\in\cA$ and every $\tau\in\cT_\beta$, $\Gamma\tau - \tau \in \cT_{<\beta}$ where $\cT_{<\beta}:= \oplus_{\zeta \in \cA_{<\beta}} \cT_\zeta$ and $\cA_{<\beta} := \cA\cap (-\infty,\beta)$.\\

In the present work, we focus on a specific regularity structure associated to the Anderson hamiltonian. We start with defining recursively two sets of symbols: let $\cU$ and $\cF$ be the smallest two sets of symbols such that $\cU$ contains all polynomials $X^k$, $k\in \N^d$, $\cF$ contains the symbol $\Xi$ (that stands for the driving noise), and we have
$$ \tau\in \cU \Longrightarrow \tau \Xi \in \cF\quad \mbox{ and }\quad \tau \in \cF \Longrightarrow \cI (\tau) \in \cU\;.$$
Note that $\tau \Xi$ and $\cI(\tau)$ should simply be seen as recursively defined symbols: the notion of models introduced below will associate some analytic features to these symbols.\\

Each symbol $\tau$ comes with a \textit{homogeneity} $|\tau|$ defined recursively as follows. We set $|\Xi| = -d/2 - \kappa$ for some arbitrary $\kappa \in (0,1/2)$, as well as $|X^k| = |k|$ for every $k\in \N^d$. For every $\tau,\tau'$, we let $|\tau\tau'| = |\tau||\tau'|$. Finally, for every $\tau\in\cF$, we set
$$ |\cI(\tau)| = |\tau| + 2 \;.$$
We let $\cA=\{|\tau|:\tau \in \cU \cup \cF\}$, and we define $\cT_\beta$ as the vector space generated by all formal linear combinations of symbols of homogeneity $\beta \in \cA$. Every such vector space is finite-dimensional and we denote by $\|\cdot\|$ the corresponding Euclidean norm.\\
Regarding the structure group $\cG$, we refer the reader to~\cite[Section 8.1]{Hairer2014} for the precise construction as the details will not be needed in the present article, except in the proof of Lemma \ref{Lemma:ModelLattice}.

We also rely on the notion of \textit{admissible model} with respect to some kernel $P$.
\begin{definition}\label{Def:AdmModel}
Let $\gamma > 0$. Let $P$ be a kernel satisfying Definition \ref{Def:Kernel} for some parameters $n_a \ge 0$ and $r > \gamma$. We say that $Z=(\Pi,\Gamma)$ is an admissible model with respect to $P$ if:\begin{itemize}
\item $\Pi = (\Pi_x, x\in\R^d)$ where for every $x\in\R^d$, $\Pi_x$ is a continuous linear map from $\cT_{<\gamma}$ into the space of distributions $\ccD'(\R^d)$ such that for any compact set $C\subset \R^d$, $\sup_{x\in C} \|\Pi\|_x < \infty$ where
$$ \| \Pi\|_x := \sup_{\varphi \in \ccB^r} \sup_{\lambda\in (0,2^{-n_a}]} \sup_{\beta \in \cA_{<\gamma}} \sup_{\tau\in\cT_\beta} \frac{|\langle \Pi_x \tau,\varphi_x^\lambda \rangle |}{\|\tau\| \lambda^\beta} \;,$$
\item $\Gamma = (\Gamma_{x,y}, x,y \in \R^d)$ such that every $\Gamma_{x,y} \in \cG$, $\Gamma_{x,z}=\Gamma_{x,y}\Gamma_{y,z}$ for every $x,y,z\in\R^d$ and for any compact set $C\subset \R^d$, $\sup_{x\in C, y\in B(x,2^{-n_a})} \|\Gamma\|_{x,y} < \infty$ where
$$ \|\Gamma\|_{x,y} := \sup_{\zeta \le \beta \in \cA_{<\gamma}} \sup_{\tau \in \cT_\beta} \frac{|\Gamma_{x,y} \tau|_\zeta}{\|\tau\| |x-y|^{\beta-\zeta}}\;,$$
\item For every $k\in\N^d$, we have
\begin{align*}
\Pi_x X^k(y) &= (y-x)^k\;,\\
\Pi_x \cI \tau (y) &= \langle \Pi_x \tau, P_+(y-\cdot)\rangle - \sum_{k\in\N^d: |k| < |\tau|+2} \frac{(y-x)^k}{k!} \langle \Pi_x \tau, \partial^k P_+(x-\cdot)\rangle\;.
\end{align*}
\end{itemize}
\end{definition}

We then set
$$ \$ \Pi \$ := \sup_{x\in (-L,L)^d} \| \Pi\|_x\;,\quad \$ \Gamma \$ := \sup_{x,y\in (-L,L)^d: |x-y| \le 2^{-n_a}} \| \Gamma\|_{x,y}\;,\quad \$Z\$ := \$ \Pi \$ + \$ \Gamma \$\;.$$

The only difference with the original setting~\cite{Hairer2014} consists in the maximal length that we impose for the analytical bounds: here we take $2^{-n_a}$ instead of $1$.

Given two admissible models $Z=(\Pi,\Gamma)$ and $\bar Z = (\bar \Pi,\bar \Gamma)$, we introduce the following distance
$$ \$ Z ; \bar Z \$ := \$ \Pi - \bar \Pi\$ + \$ \Gamma - \bar \Gamma\$\;.$$

At this point, let us provide some more concrete description of the regularity structure at stake. If $\gamma \in (3/2,2-4\kappa)$, then the symbols in $\cU$ of homogeneity less than $\gamma$, listed in increasing order of homogeneity, are:\begin{itemize}[label={--}]
\item in dimension $1$: $\un$, $X_i$, $\cI(\Xi)$; and in dimension $2$: $\un$, $\cI(\Xi)$, $X_i$,
\item in dimension $3$: $\un$, $\cI(\Xi)$, $\cI(\Xi\cI(\Xi))$, $X_i$, $\cI(\Xi\cI(\Xi\cI(\Xi)))$, $\cI(\Xi X_i)$.
\end{itemize}
The symbols in $\cF$ of homogeneity less than $\gamma-d/2-\kappa$ are obtained by multiplying the symbols above by $\Xi$.\\
Furthermore, in all dimensions we have $\Gamma_{x,y} \un = \un$, $\Gamma_{x,y} X_i = X_i + (x-y)\un$. In dimensions $2$ and $3$, we have
$$ \Gamma_{x,y} \cI(\Xi) = \cI(\Xi) + \langle \Pi_x \Xi, P_+(x-\cdot) - P_+(y-\cdot) \rangle \un\;.$$
For the remaining symbols $\tau$, the expression of $\Gamma_{x,y} \tau$ is more involved and we refer the reader to~\cite[Section 8.1]{Hairer2014}.

\bigskip

In~\cite[Prop. 3.32]{Hairer2014}, it is shown that the model is completely characterised by its action on a grid and by the knowledge of $\Gamma$. For further use, we need to push this further and show that the knowledge of $\Gamma$ is actually ``unnecessary''. First, we introduce some notation. Let $\varphi$ be the scaling function of some compactly supported wavelet basis of regularity $r> \gamma$, let $\Lambda_n := 2^n \Z^d$ and $\varphi^n_x(\cdot) := 2^{nd/2} \varphi((\cdot-x) 2^{n})$. For any admissible model $(\Pi,\Gamma)$ we then set
$$ \$ \Pi \tau \Latt := \sup_{n\ge n_a} \sup_{x\in \Lambda_n \cap (-L,L)^d} \frac{|\langle \Pi_x \tau,\varphi_x^n \rangle |}{\|\tau\| 2^{-nd/2} 2^{-n\beta}}\;,$$
together with
$$ \$ \Pi \Latt :=  \sup_{\beta \in \cA_{<0}} \sup_{\tau\in\cT_\beta}\$ \Pi \tau \Latt\;.$$

\begin{lemma}\label{Lemma:ModelLattice}
Let $\gamma > 0$. There exist two constants $C >0$ and $k \in \N$ that only depend on the kernel $P_+$ and the regularity structure $\cT_{<\gamma}$ such that the following bound holds uniformly over all admissible models $Z=(\Pi,\Gamma)$, $\bar Z =(\bar\Pi,\bar\Gamma)$
$$ \$ \Pi\$(1 + \$ \Gamma \$) \le C (1 + \$ \Pi \Latt^k)\;, \quad\$ Z;\bar Z\$ \le C \big(1+\$\bar \Pi\Latt^k\big) \big(\$ \Pi-\bar \Pi \Latt+\$ \Pi-\bar \Pi \Latt^k\big)\;.$$
\end{lemma}
The proof relies on the algebraic construction of the structure group from~\cite[Section 8]{Hairer2014}: we do not recall the notations from there since this is the only place where they are used in the present article.
\begin{proof}
Recall the sets $\cF^{(n)}$, $n\ge 0$ defined in~\cite[Subsection 8.3]{Hairer2014} by setting $\cF^{(0)} = \emptyset$ and recursively
$$ \cF^{(n+1)} := \big\{ \tau \in \cF_{F}: \Delta \tau \in \cH_F \otimes \langle \Alg(\cF^{(n)})\rangle \big\}\;.$$
As shown in~\cite[Subsection 8.3]{Hairer2014}, for every $n\ge 0$ we have
$$ \cF^{(n+1)} = \big\{ \tau \in \cF_{F}: \Delta \tau \in \langle \cF^{(n+1)} \rangle \otimes \langle \Alg(\cF^{(n)})\rangle \big\}\;,$$
and $\tau \in\Alg(\cF^{(n)}) \Rightarrow \cA \tau \in\Alg(\cF^{(n)})$ where $\cA$ is the antipode in $\cH_+$. In addition, the non-decreasing sequence $\cF^{(n)}$, $n\ge 0$ exhausts the whole set $\cF_{F} \cap \cA_{<\gamma}$ within a finite number of steps: our proof will thus rely on an induction on $n\ge 0$.\\
Recall that for any admissible model, we have
$$ \Gamma_{x,y} \tau = (I \otimes f_x\circ\cA \otimes f_y)(\Delta\otimes I)\Delta \tau\;.$$
Henceforth, for any $\tau \in \cF^{(n+1)}$, $\Gamma_{x,y} \tau$ is a linear combination of terms of the form
$$ \tau_1 f_x(\tau_2) f_y(\tau_3)\;,\quad \tau_1\in \cF^{(n+1)}, \tau_2,\tau_3 \in \Alg(\cF^{(n)})\;.$$
Recall that if $\tau' \in \Alg(\cF^{(n)})$ then $\tau' = X^k \prod_i \cJ_{\ell_i} \tau'_i$ for some $k,\ell_i \in \N^d$ and $\tau'_i \in \cF^{(n)}$ such that $|\cJ_{\ell_i} \tau'_i|>0$. Moreover, for any admissible model, we have $f_x(X^k) = (-x)^k$ and $f_x(\cJ_{\ell_i} \tau'_i) = - \langle \Pi_x  \tau'_i , \partial^{\ell_i} P_+(x-\cdot)\rangle$.\\
Putting everything together and using the notation $\$ \Pi \$_{\cF^{(n)}}$, resp.  $\$ \Gamma \$_{\cF^{(n)}}$, to denote the restriction of the norm on $\Pi$, resp. $\Gamma$, to all $\tau\in \cF^{(n)}$, we deduce that there exists an integer $k\ge 1$ such that
$$ \$ \Gamma \$_{\cF^{(n+1)}} \lesssim 1 + \$ \Pi \$_{\cF^{(n)}}^k\;,\quad \$ \Gamma - \bar\Gamma \$_{\cF^{(n+1)}} \lesssim \big(\$ \Pi - \bar \Pi \$_{\cF^{(n)}} + \$ \Pi - \bar \Pi \$_{\cF^{(n)}}^k\big)\big(1+ \$ \bar \Pi \$_{\cF^{(n)}}^k\big)\;.$$

Let us denote by $\$ \Pi \$_{\cF^{(n+1)},<\beta}$ the restriction of the norm to all elements whose homogeneity is strictly smaller than $\beta\in\R$. By~\cite[Prop 3.32]{Hairer2014} and using the fact that $\$ \Pi \$_{\Lambda,\cF^{(n+1)}}$ concerns elements with negative homogeneity, we have for every $n\ge 0$
$$ \$ \Pi \$_{\cF^{(n+1)},<0} \lesssim (1+ \$ \Gamma \$_{\cF^{(n+1)}}) \$ \Pi \$_{\Lambda,\cF^{(n+1)}}\;.$$
In the case where we are given two admissible models, this generalises into
\begin{align*}
\$ \Pi-\bar \Pi \$_{\cF^{(n+1)},<0} \lesssim (1+ \$ \Gamma \$_{\cF^{(n+1)}}) \$ \Pi-\bar \Pi \$_{\Lambda,\cF^{(n+1)}} + \$ \Gamma-\bar\Gamma \$_{\cF^{(n+1)}}\$ \bar \Pi \$_{\Lambda,\cF^{(n+1)}} \;.
\end{align*}
At this point, we observe that the inequalities collected so far ensures the following: if the bounds of the statement hold when we only consider elements in $\cF^{(n)}$, then they propagate to all elements in $\cF^{(n+1)}$ with negative homogeneity.\\
Let us finally deal with elements $\tau \in \cF^{(n+1)}$ with strictly positive homogeneity. To that end, we rely on~\cite[Prop 3.31]{Hairer2014} which ensures that if $\tau \in \cF^{(n+1)}$ is such that $|\tau| = \beta > 0$ then
$$  \sup_{x\in (-L,L)^d} \sup_{\eta \in \ccB^r} \sup_{\lambda\in (0,2^{-n_a}]}\frac{|\langle \Pi_x \tau,\eta_x^\lambda \rangle |}{\|\tau\| \lambda^\beta} \lesssim \$\Pi \$_{\cF^{(n+1)},<\beta} \$ \Gamma\$_{\cF^{(n+1)}}\;,$$
and
\begin{align*}
\sup_{x\in (-L,L)^d}\sup_{\eta \in \ccB^r} \sup_{\lambda\in (0,2^{-n_a}]}\frac{|\langle \Pi_x -\bar\Pi_x \tau,\eta_x^\lambda \rangle |}{\|\tau\| \lambda^\beta} &\lesssim \$\Pi - \bar \Pi \$_{\cF^{(n+1)},<\beta} \$ \Gamma\$_{\cF^{(n+1)}}\\
&+ \$\Gamma - \bar \Gamma \$_{\cF^{(n+1)}} \$ \bar\Pi\$_{\cF^{(n+1)},<\beta}\;,
\end{align*}
Applying these bounds first to the element with the smallest positive homogeneity in $\cF^{(n+1)}$ and then, inductively in increasing order of homogeneities, one deduces that there exists $k',k''\ge 0$ such that the following bounds hold
$$ \$ \Pi \$_{\cF^{(n+1)}} \lesssim (1+ \$ \Gamma \$_{\cF^{(n+1)}})^{k'} \$ \Pi \$_{\Lambda,\cF^{(n+1)}}\;,$$
and 
\begin{align*}
\$ \Pi-\bar \Pi \$_{\cF^{(n+1)}} \lesssim (1+ \$ \Gamma \$_{\cF^{(n+1)}})^{k''}\Big( \$ \Pi-\bar \Pi \$_{\Lambda,\cF^{(n+1)}} + \$ \Gamma-\bar\Gamma \$_{\cF^{(n+1)}} (1+ \$ \bar\Gamma \$_{\cF^{(n+1)}})^{k'}  \$ \bar \Pi \$_{\Lambda,\cF^{(n+1)}}\Big) \;.
\end{align*}
A simple recursion on $n\ge 0$ then completes the proof (note that the integer $k$ of the statement may be larger than the one introduced above).
\end{proof}

\subsection{Spaces of modelled distributions}

Given an admissible model $(\Pi,\Gamma)$ and a parameter $a\ge 1$, we introduce the following notations. We let $L^p_x := L^p(\R^d,dx)$. We let $P := \partial (-L,L)^d$ be the boundary of the hypercube $(-L,L)^d$, and we set
$$ |x|_P := \frac1{\sqrt a} \wedge \dist(x,P)\;.$$

\begin{definition}
Take $\gamma >0$, $\sigma \le \gamma$ and $p \in [1,\infty]$. Let $\cD^{\gamma,\sigma}_{p,\infty}$ be the Banach space of all maps $f:\R^d\mapsto \cT_{<\gamma}$ such that $f$ vanishes outside $[-L,L]^d$ and
\begin{align*}
\$ f\$_{\gamma,\sigma} &:= \sum_{\zeta\in \cA_{<\gamma}} \Big\| \frac{|f(x)|_\zeta}{|x|_{P}^{\sigma-\zeta}}\Big\|_{L^p_x}+ \sum_{\zeta\in \cA_{<\gamma}} \sup_{h\in\R^d:|h|<1/\sqrt a} \Big\| \frac{\big| f(x+h) - \Gamma_{x+h,x} f(x) \big|_\zeta}{|x|_P^{\sigma-\gamma} |h|^{\gamma-\zeta}} \un_{\{|x|_P > 3|h|\}}\Big\|_{L^p_x} < \infty\;.
\end{align*}
\end{definition}

Note that this space is in the flavour of the classical Besov spaces $\cB_{p,\infty}^\gamma$. Actually if we disregard the weights near the boundary $P$ and if the regularity structure only consists of polynomials, then these two spaces coincide as long as $\gamma$ is non-integer.\\
In the original version of the theory of regularity structures~\cite{Hairer2014}, the spaces of modelled distributions are endowed with H\"older-type norms: this corresponds to taking $p=\infty$ here. An extension of the analytical part of the theory to the whole class of Besov-type spaces was presented in~\cite{Recons}.

\medskip

The parameter $\sigma$ is the exponent of the weight that we impose near the boundary of the domain. This weight allows for an explosion of all components of the modelled distribution that lie at levels higher than $\sigma$, but forces a decay of the other components. Such weights on the time coordinate already appeared in the parabolic setting~\cite[Section 6]{Hairer2014} to allow for irregular initial conditions (and to obtain contractivity of the fixed point map). Recently, Gerencs\'er and Hairer~\cite{Mate} introduced a more general framework of modelled distributions weighted near time $0$ and near the boundary of the spatial domain: in this setting, they were able to solve singular SPDEs in domains with boundaries.

\begin{remark}[Periodic]
In the periodic case, the definition becomes the following. Take $\gamma >0$ and $p \in [1,\infty]$. Let $\cD^{\gamma}_{p,\infty}$ be the Banach space of all maps $f:\R^d\mapsto \cT_{<\gamma}$ such that $f$ is periodic and
\begin{align*}
\sum_{\zeta\in \cA_{<\gamma}} \Big\| |f(x)|_\zeta \Big\|_{L^p_x((-L,L)^d)}+ \sum_{\zeta\in \cA_{<\gamma}} \sup_{h\in\R^d:|h|<1/\sqrt a} \Big\| \frac{\big| f(x+h) - \Gamma_{x+h,x} f(x) \big|_\zeta}{|h|^{\gamma-\zeta}}\Big\|_{L^p_x((-L,L)^d)} < \infty\;.
\end{align*}
\end{remark}

We introduce the notations $\cQ_\zeta f$ and $f_\zeta$ to denote the projection of $f$ on $\cT_\zeta$. We also let $\langle f, \tau\rangle$ be the projection of $f$ on the vector space generated by the symbol $\tau$. We conclude this subsection with introducing a distance between modelled distributions associated with two distinct models $(\Pi,\Gamma)$, $(\bar\Pi,\bar\Gamma)$:
\begin{align*}
\$f;\bar f\$ &:=  \sum_{\zeta\in \cA_{<\gamma}} \Big\| \frac{|f(x)-\bar f(x)|_\zeta}{|x|_{P}^{\sigma-\zeta}}\Big\|_{L^p_x}\\
&+ \sum_{\zeta\in \cA_{<\gamma}} \sup_{h\in\R^d:|h|<1/\sqrt a}\Big\| \frac{\big| f(x+h) - \bar f(x+h) - \Gamma_{x+h,x} f(x) + \bar\Gamma_{x+h,x} \bar f(x) \big|_\zeta}{|x|_P^{\sigma-\gamma} |h|^{\gamma-\zeta}} \un_{\{|x|_P > 3|h|\}}\Big\|_{L^p_x}
\end{align*}

\subsection{Reconstruction}

Let $\cB^\alpha_{p,\infty}(D)$ be the usual Besov space of regularity index $\alpha$ on some regular domain $D\subset\R^d$. The goal of this subsection is to construct a continuous linear operator $\cR$ that maps in a consistent way elements in $\cD^{\gamma,\sigma}_{p,\infty}$ to genuine distributions. Consistency here simply means that the local behaviour of the distribution $\cR f$ around some point $x$ should match $\Pi_x f(x)$ up to some negligible error term. Such a result is usually referred to as a \textit{reconstruction theorem}~\cite[Th. 3.10]{Hairer2014}.

In the unweighted setting, a reconstruction theorem in Besov-type spaces of modelled distributions was established in~\cite[Th 3.1]{Recons}. Since reconstruction is a local operation and since away from the boundary $P$ our weighted spaces of modelled distributions are equivalent with their unweighted counterparts, the aforementioned result allows to associate to any $f\in\cD^{\gamma,\sigma}_{p,\infty}$ an element $\tilde{\cR} f$ in the dual space of all test functions supported away from $P$ and such that $f\mapsto\tilde{\cR} f$ is a continuous, linear map from $\cD^{\gamma,\sigma}_{p,\infty}$ into $\cB^{\alpha}_{p,\infty}(D)$ for any compact set $D\subset (-L,L)^d$ and where $\alpha := \min(\cA\backslash \N)\wedge\gamma$. Furthermore, we have the bound
\begin{equation}\label{Eq:ReconsBound}
\sup_{\lambda \in (0,1/\sqrt a)} \Big\| \sup_{\varphi \in \ccB^r} \frac{|\langle \tilde{\cR} f - \Pi_x f(x) , \varphi^\lambda_x\rangle|}{|x|_P^{\sigma-\gamma}\lambda^\gamma} \un_{\{|x|_P > 3 \lambda\}} \Big\|_{L^p_x} \lesssim \$f\$\;,
\end{equation}
uniformly over all $f\in\cD^{\gamma,\sigma}_{p,\infty}$.\\
In the case where we deal with two models $(\Pi,\Gamma)$, $(\bar\Pi,\bar\Gamma)$, we have the bound
\begin{equation}\label{Eq:ReconsBound2}
\sup_{\lambda \in (0,1/\sqrt a)} \Big\| \sup_{\varphi \in \ccB^r} \frac{|\langle \tilde{\cR} f - \bar{\tilde{\cR}} \bar f - \Pi_x f(x) + \Pi_x \bar f(x) , \varphi^\lambda_x\rangle|}{|x|_P^{\sigma-\gamma}\lambda^\gamma} \un_{\{|x|_P > 3 \lambda\}} \Big\|_{L^p_x}  \lesssim \$f;\bar f\$\;,
\end{equation}
uniformly over all $f\in\cD^{\gamma,\sigma}_{p,\infty}$ and all $\bar f \in \bar{\cD}^{\gamma,\sigma}_{p,\infty}$.\\

The following result extends the reconstruction operator up to the boundary $P$ of the domain.
\begin{theorem}[Reconstruction]\label{Th:Recons}
Take $\gamma \in \R_+\backslash \N$, $\sigma\le \gamma$ and set $\alpha := \min(\cA \backslash \N) \wedge \gamma$. Assume that
$$ -1+ \frac1{p} < \sigma \le \gamma\;.$$
Set $\bar\alpha = \alpha\wedge\sigma$. There exists a unique continuous linear operator $\cR:\cD^{\gamma,\sigma}_{p,\infty}\to \cB^{\bar\alpha}_{p,\infty}(\R^d)$ such that $\cR f = \tilde{\cR} f$ away from $P$, $\cR f = 0$ outside $(-L,L)^d$ and for any $C>0$ we have
\begin{equation}\label{Eq:BoundReconsBdry}
\sup_{\lambda \in (0,1/\sqrt a)} \Big\| \sup_{\eta \in \ccB^r} \Big| \frac{\langle \cR f , \eta^\lambda_x\rangle}{\lambda^\sigma} \Big| \un_{\{|x|_P \le C \lambda\}} \Big\|_{L^p_x} \lesssim \$ f \$\;,
\end{equation}
uniformly over all $f\in \cD^{\gamma,\sigma}_{p,\infty}$ and all $L\ge 1$. In the case where we deal with two models $(\Pi,\Gamma)$ and $(\bar\Pi,\bar\Gamma)$, the same holds with $\cR f$ replaced by $\cR f-\bar\cR \bar f$ and with $\$ f \$$ replaced by $ \$ f ;\bar f \$$.
\end{theorem}
\begin{remark}[Periodic]
In the periodic case, \eqref{Eq:BoundReconsBdry} remains true if $\lambda^\sigma$ is replaced by $\lambda^\alpha$, $L^p_x$ is replaced by $L^p_x((-L,L)^d)$ and the indicator is removed. This situation is already covered by~\cite[Th 3.1]{Recons}.
\end{remark}
\begin{remark}
Note that the Besov norm of $\cR f$ is bounded by $\$f\$$ times a constant uniformly bounded over all $f\in \cD^{\gamma,\sigma}_{p,\infty}$. However this constant is \textit{not} uniformly bounded over all $a\ge 1$.
\end{remark}

The condition $-1+1/p < \sigma$ ensures that the weights near the boundary do not explode too fast and allow to define a unique distribution (recall that a Dirac mass has regularity $-1+1/p$ in the Besov scale $\cB_{p,\infty}$). The reader familiar with the reconstruction theorem would probably have expected the more restrictive condition $-1+1/p < \alpha\wedge\sigma$. Here, we are actually able to deal with the case where $\alpha$ is smaller than $-1+1/p$ as long as $\sigma$ remains strictly above this threshold. This is in the same flavour as~\cite[Th 4.10]{Mate} though our proof is different. The main technical step to establishing Theorem \ref{Th:Recons} is the following extension result.
\begin{lemma}\label{Lemma:Recons}
Let $\sigma > -1 + 1/p$. Let $\xi$ be a distribution on $\R^d\backslash P$ such that for any $c>0$ we have
\begin{equation}\label{Eq:ReconsNearBdry}
\sup_{\lambda \in (0,1/\sqrt a)} \Big\| \sup_{\eta \in \ccB^r} \frac{\big|\langle \xi, \eta^\lambda_x \rangle\big|}{\lambda^\sigma} \un_{\{3\lambda < |x|_P < c \lambda\}}\Big\|_{L^p_x} < \infty\;.
\end{equation}
Then, $\xi$ admits a unique extension into a distribution on $\R^d$ that satisfies for all $C>0$ the bound
\begin{equation}\label{Eq:ReconsNearBdry2}
\sup_{\lambda \in (0,1/\sqrt a)} \Big\| \sup_{\eta \in \ccB^r} \frac{\big|\langle \xi, \eta^\lambda_x \rangle\big|}{\lambda^\sigma} \un_{\{|x|_P < C \lambda\}}\Big\|_{L^p_x} < \infty\;.
\end{equation}
\end{lemma}
With this lemma at hand, the proof of the reconstruction theorem goes as follows.
\begin{proof}[Proof of Theorem \ref{Th:Recons}]
We aim at applying Lemma \ref{Lemma:Recons}: to that end, it suffices to check that $\tilde{\cR} f$ satisfies \eqref{Eq:ReconsNearBdry}. Fix $c>0$. By the reconstruction bound \eqref{Eq:ReconsBound}, we have
\begin{align*}
\sup_{\lambda \in (0,1/\sqrt a)} \Big\| \sup_{\eta \in \ccB^r} \frac{\big|\langle \tilde{\cR} f - \Pi_x f(x) , \eta^\lambda_x \rangle\big|}{\lambda^{\sigma}} \un_{\{3\lambda < |x|_P < c \lambda\}}\Big\|_{L^p_x} \lesssim \$f\$\;.
\end{align*}
Furthermore, we have
\begin{align*}
\sup_{\lambda \in (0,1/\sqrt a)} \Big\| \sup_{\eta \in \ccB^r} \frac{\big|\langle \Pi_x f(x) , \eta^\lambda_x \rangle\big|}{\lambda^\sigma} \un_{\{3\lambda < |x|_P < c \lambda\}}\Big\|_{L^p_x} &\lesssim \sum_{\zeta \in \cA_{<\gamma}} \sup_{\lambda \in (0,1/\sqrt a)} \Big\| \sup_{\eta \in \ccB^r} \frac{\big|f_\zeta(x)\big|}{\lambda^{\sigma-\zeta}} \un_{\{3\lambda < |x|_P < c \lambda\}}\Big\|_{L^p_x}\\
&\lesssim \sum_{\zeta \in \cA_{<\gamma}} \sup_{\lambda \in (0,1/\sqrt a)} \Big\| \sup_{\eta \in \ccB^r} \frac{\big|f_\zeta(x)\big|}{|x|_P^{\sigma-\zeta}} \un_{\{3\lambda < |x|_P < c \lambda\}}\Big\|_{L^p_x}\\
&\lesssim \$f\$\;.
\end{align*}
Applying Lemma \ref{Lemma:Recons}, we obtain an extension of $\tilde{\cR} f$ that we denote $\cR f$: from \eqref{Eq:ReconsNearBdry2} and the regularity of $\tilde{\cR} f$ we deduce that $\cR f$ belongs to $\cB^{\bar\alpha}_{p,\infty}(\R^d)$ and that it satisfies \eqref{Eq:BoundReconsBdry}. The case of two models follows from similar arguments.
\end{proof}
\begin{proof}[Proof of Lemma \ref{Lemma:Recons}]
Let $\chi$ be a smooth, compactly supported function that defines a partition of unity: $\sum_{n\in\Z} \chi(x 2^n) = 1$ for all $x > 0$. The proof consists of two steps.\\
\textit{Uniqueness.} Let $\xi$ be an extension. Let us introduce a smooth indicator of the $2^{-n_0}$-neighbourhood of $P$ by setting $J_{n_0}(x) := 1 - \sum_{n\le n_0} \chi(|x|_P 2^n)$ for all $x\in\R^d$ and every $n_0\in\Z$. Suppose that
\begin{equation}
\big|\langle \xi,\varphi_{x_0} J_{n_0}\rangle\big| \lesssim 2^{-n_0(1-\frac1{p} + \sigma)}\;,
\end{equation}
uniformly over all $x_0\in P$, all $n_0\in\N$ large enough and all $\varphi\in\ccB^r$. Since $1-1/p+\sigma >0$, we deduce that $\xi$ is completely characterised by its evaluations away from $P$ so that the extension is necessarily unique. We are left with establishing the asserted bound. To that end, we consider a smooth function $\psi$ supported in the unit ball that defines a partition of unity:
$$ \sum_{x\in\Lambda_0} \psi_x(y) = 1\;,\quad y\in\R^d\;,$$
where $\psi_x(y) := \psi(y-x)$ and $\Lambda_n := \{(2^{-n} k_1,\ldots,2^{-n} k_d): k_i \in\Z\}$ for all $n\in\Z$. We also define $\psi^n_x(y) := \psi(2^n(y-x))$ and obtain a partition of unity at scale $2^{-n}$:
$$ \sum_{x\in\Lambda_n} \psi_x^n(y) = 1\;,\quad y\in\R^d\;.$$
We thus write
$$ \varphi_{x_0} J_{n_0} = \sum_{x\in\Lambda_{n_0}} \varphi_{x_0} J_{n_0} \psi_x^{n_0}\;,$$
and we observe that the only non-zero contributions in this series come from $x$'s such that $|x|_P \le C 2^{-n_0}$ for some constant $C>0$ independent of $n_0$. Furthermore, for any such $x$ and for any $y\in B(x,2^{-n_0})$, each function $\varphi_{x_0} J_{n_0} \psi_x^{n_0}$ can be rewritten as $2^{-n_0 d} \eta_y^{2^{-(n_0-1)}}$ for some $\eta \in \ccB^r$, up to a uniformly bounded multiplicative constant. Therefore,
\begin{align*}
\big| \langle \xi , \varphi_{x_0} J_{n_0}\rangle \big| &\lesssim \sum_{x\in\Lambda_{n_0}:|x|_P \le C 2^{-n_0}}\fint_{y\in B(x,2^{-n_0})}\sup_{\eta\in\ccB^r} 2^{-n_0 d} \big| \langle \xi , \eta_y^{2^{-(n_0-1)}} \rangle\big| dy\\
&\lesssim \Big(\sum_{x\in\Lambda_{n_0}:|x|_P \le C 2^{-n_0}} 2^{-n_0 (d-1)}\fint_{y\in B(x,2^{-n_0})}\sup_{\eta\in\ccB^r} 2^{-n_0 p} \big| \langle \xi , \eta_y^{2^{-(n_0-1)}} \rangle\big|^p dy\Big)^{\frac1{p}}\\
&\lesssim 2^{-n_0(1-\frac1{p})} \Big(\sum_{x\in\Lambda_{n_0}:|x|_P \le C 2^{-n_0}}\int_{y\in B(x,2^{-n_0})}\sup_{\eta\in\ccB^r}\big| \langle \xi , \eta_y^{2^{-(n_0-1)}} \rangle\big|^p dy\Big)^{\frac1{p}}\\
&\lesssim 2^{-n_0(1-\frac1{p}+\sigma)} \Big(\int_{y\in \R^d: |y|_P \le (C+1) 2^{-n_0}} \Big(\sup_{\eta\in\ccB^r}\frac{\big| \langle \xi , \eta_y^{2^{-(n_0-1)}} \rangle\big|}{2^{-n_0 \sigma}}\Big)^p dy\Big)^{\frac1{p}}\;,
\end{align*}
uniformly over all $n_0\ge 0$, all $x_0\in P$ and all $\varphi\in\ccB^r$. Since we assumed that $\xi$ is an extension, it satisfies \eqref{Eq:ReconsNearBdry2} and therefore the second factor on the r.h.s.~is bounded uniformly over all $n_0$ large enough. This concludes the proof of uniqueness.\\
\textit{Existence.} We define $\varphi^n_x(y) := \chi(2^{n}|y|_P) \psi^n_x(y)$ for all $n\in\Z$, all $x\in\Lambda_n$ and all $y\in\R^d$. We then have
$$ \sum_{n\in\Z} \sum_{x\in\Lambda_n} \varphi^n_x(y) = 1\;,\quad y\in\R^d\backslash P\;.$$
For any compactly supported, smooth function $\eta$, we set
$$ \langle \hat{\xi} , \eta \rangle := \sum_{n\in\Z} \sum_{y\in\Lambda_n} \langle \xi , \varphi^n_y \eta\rangle\;.$$
When the support of $\eta$ does not intersect $P$, one easily checks that this coincides with $\langle \xi,\eta\rangle$. If we show that \eqref{Eq:ReconsNearBdry2} holds with $\hat{\xi}$ instead of $\xi$, then standard arguments show that $\hat{\xi}$ is indeed a distribution that extends $\xi$.\\
Fix $C>0$ and consider $\eta^\lambda_x$ with $|x|_P < C \lambda$. Denote by $[b,b']$ the support of $\chi$. If $b$ is chosen large enough compared to $C$, then $\varphi^n_y \eta^\lambda_x$ vanishes whenever $2^{-n} > \lambda$. Moreover, $\varphi^n_y \eta^\lambda_x$ vanishes whenever $|x-y| > 2^{-n}+\lambda$ or $|y|_P \notin [(b-1) 2^{-n},(b'+1) 2^{-n}]$. Finally, when $2^{-n} \le \lambda$, for any $z\in B(y,2^{-n})$ the function $\varphi^n_y \eta^\lambda_x$ can be written as $2^{-nd} \lambda^{-d} \rho_z^{2^{-(n-1)}}$ for some $\rho\in\ccB^r$ up to a uniformly bounded multiplicative constant. Hence we have
\begin{align*}
&\sup_{\lambda \in (0,1/\sqrt a)} \Big\| \sup_{\eta \in \ccB^r} \frac{\big|\langle \xi, \eta^\lambda_x \rangle\big|}{\lambda^\sigma} \un_{\{|x|_P < C \lambda\}}\Big\|_{L^p_x}\\
&\lesssim \sup_{\lambda \in (0,1/\sqrt a)}\bigg\| \sum_{n: 2^{-n} \le \lambda}\!\!\!\!\!\!\! \sum_{\substack{y\in\Lambda_n:\\|y-x| \le \lambda+2^{-n}\\ (b-1) 2^{-n} < |y|_P < (b'+1) 2^{-n}}}\!\!\!\! \fint_{z\in B(y,2^{-n})} 2^{-nd} \lambda^{-d} \sup_{\rho \in \ccB^r} \frac{\big|\langle \xi, \rho_z^{2^{-(n-1)}} \rangle\big|}{\lambda^\sigma}dz \un_{\{|x|_P < C \lambda\}}\bigg\|_{L^p_x}\;.
\end{align*}
Notice that for every $x\in\R^d$, the sum over $y$ contains a number of non-zero terms which is at most of order $(\lambda 2^n)^{d-1}$. Using the triangle inequality on the sum over $n$ and Jensen's inequality on the sum over $y$ and the integral over $z$, we get
\begin{align*}
&\lesssim \sup_{\lambda \in (0,1/\sqrt a)} \sum_{n: 2^{-n} \le \lambda} \Big( \int_{x\in\R^d} \sum_{\substack{y\in\Lambda_n:\\|y-x| \le \lambda+2^{-n}\\ (b-1) 2^{-n} < |y|_P < (b'+1) 2^{-n}}} \fint_{z\in B(y,2^{-n})} 2^{-n(d-1)} \lambda^{-(d-1)}\\
&\qquad\qquad\qquad\times\Big(\sup_{\rho \in \ccB^r} 2^{-n} \lambda^{-1}\frac{\big|\langle \xi, \rho_z^{2^{-(n-1)}} \rangle\big|}{\lambda^\sigma} \un_{\{|x|_P < C \lambda\}}\Big)^p dz\,dx\Big)^{\frac1{p}}\\
&\lesssim \sup_{\lambda \in (0,1/\sqrt a)} \sum_{n: 2^{-n} \le \lambda} 2^{-n} \lambda^{-1}\Big( \int_{z\in\R^d:(b-2)2^{-n} < |z|_P < (b'+2) 2^{-n}} 2^{n} \lambda \Big(\sup_{\rho \in \ccB^r} \frac{\big|\langle \xi, \rho_z^{2^{-(n-1)}} \rangle\big|}{\lambda^\sigma} \Big)^p dz\Big)^{\frac1{p}}\\
&\lesssim \sup_{\lambda \in (0,1/\sqrt a)} \sum_{n: 2^{-n} \le \lambda} \big(2^{-n}\lambda^{-1}\big)^{(1-\frac1{p}+\sigma)} \Big\|\sup_{\rho \in \ccB^r} \frac{\big|\langle \xi, \rho_z^{2^{-(n-1)}} \rangle\big|}{2^{-n\sigma}} \un_{\{(b-2)2^{-n} < |z|_P < (b'+2) 2^{-n}\}} \Big\|_{L^p_z}\\
&\lesssim \sup_{n\ge n_a} \Big\|\sup_{\rho \in \ccB^r} \frac{\big|\langle \xi, \rho_z^{2^{-(n-1)}} \rangle\big|}{2^{-n\sigma}} \un_{\{(b-2)2^{-n} < |z|_P < (b'+2) 2^{-n}\}} \Big\|_{L^p_z}\;.
\end{align*}
Using the bound \eqref{Eq:ReconsNearBdry} and choosing $b$ large enough, we deduce that the latter quantity is finite thus concluding the proof.
\end{proof}

\section{Convolution and fixed point}\label{Sec:Convol}

\subsection{The Green's function}

For $a\ge 1$, let $P^{(a)}$ be the Green's function of $-\Delta + a$ on $\R^d$, that is, the solution\footnote{As usual, since this equation does not have a unique solution we further impose that $P^{(a)}$ remains bounded at infinity and this reduces the set of solutions to a single element.} of
$$ (-\Delta+a) P^{(a)} = \delta_0\;,$$
in the sense of distributions. It is well-known that
$$ P^{(a)}(x) =\begin{cases} \frac1{2\sqrt a}\, e^{-\sqrt a |x|}\;,&\quad d=1\;,\\
\frac1{2\pi}\,\Bess(\sqrt a |x|)\;,&\quad d=2\;,\\
\frac1{4\pi |x|}\, e^{-\sqrt a |x|}\;,&\quad d=3\;,
\end{cases}$$
where $\Bess$ is the modified Bessel function of the second kind of index $0$ (usually denoted $K_0$, see~\cite[Section 9.6]{AbraSteg}). Recall that for every $k\ge 0$ we have
$$ \partial^k \Bess(x) \sim (-1)^k \sqrt{\frac{\pi}{2x}} e^{-x}\;,\quad x\to\infty\;,$$
and for every $k\ge 1$
$$ \Bess(x) \sim -\ln x\;,\quad \partial^k \Bess(x) \sim \frac{(-1)^{k} (k-1)!}{x^k}\;,\quad x\to 0_+\;.$$
From now on, we let $n_a$ be the smallest integer such that $2^{-n_a} \le 1/\sqrt a$.
\begin{lemma}
There exists a decomposition $P_n^{(a)}, n\ge n_a$ and $P_-^{(a)}$ that satisfies Definition \ref{Def:Kernel} and such that for all $k\in\N^d$ such that $|k| < r$ we have
\begin{equation}\label{Eq:BndP}
\big| \partial^k P_n^{(a)}(x)\big| \lesssim 2^{-n(2-d-|k|)}\;,\quad \sup_{y\in B(x,1/\sqrt a)} \big| \partial^k P_-^{(a)}(y) \big| \lesssim 2^{-n_a(2-d-|k|)} e^{-\sqrt a (|x|-\frac1{\sqrt a})_+}\;,
\end{equation}
uniformly over all $n\ge n_a$, all $a\ge 1$ and all $x\in\R^d$.
\end{lemma}
\noindent In Subsection \ref{Sec:Models}, we will define a canonical model $(\Pi_\epsilon^{(a)},\Gamma_\epsilon^{(a)})$ based on the kernel $P^{(a)}$ and some regularised white noise $\xi_\epsilon$.
\begin{proof}
Let $\varphi:[0,\infty)\to[0,1]$ be a smooth function, supported in $[1/4,1]$, that defines a partition of unity in the following sense:
$$ \sum_{n\in \Z} \varphi(2^n x) = 1\;,\quad x\in (0,\infty)\;.$$
Our construction of the decomposition of $P^{(a)}$ into a sum of smooth functions depends on the dimension. For $d=3$, we set for every $n\ge n_a$
$$ \bar{P}_n^{(a)}(x) := \varphi(2^n |x|) P^{(a)}(x)\;,\quad \bar{P}_-^{(a)}(x) := \sum_{n < n_a} \varphi(2^n |x|) P^{(a)}(x)\;,\quad x\in \R^3\;.$$
We now modify these functions in order to get property 4. of Definition \ref{Def:Kernel}. Set for every $k\in \N^d$ and $n\ge n_a$
$$ I_{k,n}^{(a)} := \sum_{\ell \ge n} \int x^k \bar{P}_\ell^{(a)}(x) dx\;,$$
and observe that $|I_{k,n}^{(a)}| \lesssim 2^{-n(2+|k|)}$ uniformly over all $a\ge 1$ and all $n\ge n_a$. Consider a smooth function $\eta_k$ supported in $B(0, 1/2) \subset \R^d$ and such that $\int x^\ell \eta_k(x) dx = \delta_{k,\ell}$ for every $\ell\in \N^d$ such that $|\ell| < r$. Set also $\eta_{k,n}(\cdot) := 2^{n(d + |k|)} \eta_k(2^n\cdot)$. Then, we set for every $n \ge n_a$
$$ P_{n}^{(a)} :=  \bar{P}_{n}^{(a)} + \sum_{k:|k| < r} \Big(I_{k,n+1}^{(a)} \eta_{k,n+1} - I_{k,n}^{(a)} \eta_{k,n}\Big)\;,$$
and
$$ P_-^{(a)} := \bar{P}_-^{(a)} + \sum_{k:|k| < r} I_{k,n_a}^{(a)} \eta_{k,n_a}\;.$$
It is simple to check that the requirements of Definition \ref{Def:Kernel} are satisfied, together with the bounds \eqref{Eq:BndP}.\\
For $d=2$, we set for every $n\ge n_a$
$$ \bar{P}_n^{(a)}(x) := - \frac1{2\pi} \int_{|x|}^\infty \sqrt a\, \Bess'(\sqrt a r) \varphi(2^n r) dr \;,\quad x\in \R^2\;,$$
and
$$ \bar{P}_-^{(a)}(x) := \frac1{2\pi}\Bess(\sqrt a |x|) - \sum_{n\ge n_a} \bar{P}_n^{(a)}(x)\;,\quad x\in \R^2\;.$$
Then, we repeat the same procedure as for $d=3$ to annihilate polynomials. Using the asymptotic behaviour of the modified Bessel function recalled before the lemma, we deduce that the asserted bounds hold.\\
We now consider the case $d=1$. Let $\psi:\R\to [0,1]$ be a smooth function, supported in the unit ball and such that $\psi(x) = 1$ whenever $|x| \le 1/2$. We then set
$$\bar{P}_{n_a}^{(a)}(x) := \psi(2^{n_a} |x|) P^{(a)}(x)\;,\quad x\in \R\;,$$
as well as $\bar{P}_n^{(a)} := 0$ for all $n > n_a$, and $\bar{P}_-^{(a)} := P^{(a)} - \bar{P}_{n_a}^{(a)}$. We then follow the same procedure as for $d=3$ in order to annihilate polynomials.
\end{proof}

Let us now introduce a convenient representation of the Green's function $K^{(a)}$ associated to $-\Delta + a$ on $(-L,L)^d$ with Dirichlet b.c. The reflection principle allows to construct $K^{(a)}$ as a series of shifted versions of $P^{(a)}$. More precisely, if for all $m\in\Z^d\backslash\{0\}$ we set $\epsilon_m := (-1)^{|m|}$ then we have
\begin{equation}\label{Eq:KaPa}
K^{(a)}(x,y) = P^{(a)}(x,y) + \sum_{m\in\Z^d\backslash\{0\}} \epsilon_m P^{(a)}(x,\pi_{m,L}(y))\;,\quad x,y \in \R^d\;,
\end{equation}
where $\pi_{m,L}$ is the bijection from $\R^d$ to $\R^d$ defined as follows:
$$ \big(\pi_{m,L}(y)\big)_i = (-1)^{m_i} y_i + 2L m_i\;.$$
Then, for any map $u$ that vanishes outside $[-L,L]^d$, the function
$$ f(x) := \int_{\R^d} K^{(a)}(x,y)u(y) dy\;,\quad x\in(-L,L)^d\;,$$
is the solution of $(-\Delta + a) f = u$ on $(-L,L)^d$ endowed with Dirichlet b.c.

\medskip

It is then natural to introduce
$$ K^{(a)}_{-}(x,y) := P^{(a)}_{-}(x,y) + \sum_{m\in\Z^d\backslash\{0\}} \epsilon_m P^{(a)}_-(x,\pi_{m,L}(y))\;,$$
and for every $n\ge n_a$
$$ Z^{(a)}_n(x,y) := \sum_{m\in\Z^d\backslash\{0\}} \epsilon_m P^{(a)}_n(x,\pi_{m,L}(y))\;.$$
Notice that $K^{(a)}_-$ is a smooth function. We set $Z^{(a)}_+ := \sum_{n\ge n_a} Z^{(a)}_n$. We thus come to the following decomposition:
\begin{equation}\label{Eq:DecompKa}
K^{(a)}(x,y) = P^{(a)}_+(x,y) + Z^{(a)}_+(x,y) + K^{(a)}_-(x,y)\;,\quad x,y \in \R^d\;.
\end{equation}
It will also be convenient to define $K^{(a)}_n(x,y) = P^{(a)}_n(x,y) + Z^{(a)}_n(x,y)$ together with $K^{(a)}_+ = \sum_{n\ge n_a} K^{(a)}_n$.\\

Since we will convolve this kernel with distributions that vanish outside $[-L,L]^d$, it is convenient to make this kernel compact in the $y$ variable for all $x\in [-L,L]^d$. To that end, we let $\chi_n$ be a smooth function that equals $1$ on $[-L-2^{-n},L+2^{-n}]^d$ and $0$ on the complement of $[-L-2\cdot 2^{-n},L+2\cdot 2^{-n}]^d$, and we replace $K_n^{(a)}(x,y)$ by $K_n^{(a)}(x,y) \chi_n(y)$ and $K_-^{(a)}(x,y)$ by $K_-^{(a)}(x,y) \chi_{n_a}(y)$.\\
For all $x\in (-L,L)^d$ we then have $P^{(a)}_n(x,y)\chi_n(y) = P^{(a)}_n(x,y)$, while
$$ Z^{(a)}_n(x,y)\chi_n(y) = \sum_{m\in\Z^d\backslash\{0\}:\sup_i |m_i| \le 1} \epsilon_m P^{(a)}_n(x,\pi_{m,L}(y))\;.$$
Note also that for all $x\in [-L,L]^d$, $Z^{(a)}_n(x,\cdot)\chi_n(\cdot)$ is null except if $|x|_P \le 3\cdot 2^{-n}$. From now on, we assume that all the above kernels are implicitly multiplied by this cutoff function.

\begin{remark}[Periodic]
In the periodic case, it suffices to set $K^{(a)}(x,y) := P^{(a)}(x-y)$. To deal with unified notations, one can take $\epsilon_m = 0$ for all $m\in\Z^d\backslash\{0\}$ so that the identity \eqref{Eq:KaPa} still holds.
\end{remark}
\begin{remark}
In the Neumann case, the above construction still applies with $\epsilon_m = 1$.
\end{remark}


The next lemma ``quantifies'' the effect of the Dirichlet boundary conditions on the kernel $K^{(a)}_+$. A similar statement holds for $K^{(a)}_-$.

\begin{lemma}\label{Lemma:CancelDirichlet}
There exists $c>0$ such that for all $a,L\ge 1$, all $x\in [-L,L]^d$ and all $n\ge n_a$ such that $3\cdot 2^{-n} \ge |x|_P$, there exists a smooth function $\varphi\in\ccB^r$ such that $K_n^{(a)}(x,\cdot)$ can be written as $2^{-n} |x|_P \varphi_x^{c2^{-n}}(\cdot)$ up to a multiplicative constant which is bounded uniformly over all the parameters.
\end{lemma}
Note that in general, one has a prefactor $2^{-2n}$ in front of $\varphi_x^{c2^{-n}}(\cdot)$. The lemma shows that near the boundary, one can trade a factor $2^{-n}$ off for a factor $|x|_P$.
\begin{proof}
Let $x\in [-L,L]^d$ and $n\ge n_a$. Assume that $x$ is at distance larger than $3\cdot 2^{-n}$ from the hyperplanes $x_i =\pm L$ for $i>1$ and lies at distance less than $3\cdot 2^{-n}$ from $x_1=L$. We let $x_P$ be its projection on the hyperplane $x_1 = L$. If we denote by $e_i$ the element of $\Z^d$ whose coordinates are all zero except the $i$-th which equals $1$, then $\pi_{e_1,L}$ is the reflection with respect to the hyperplane $x_1=L$. Recall that we only ever consider $y$'s in $[-L-2\cdot 2^{-n},L+2\cdot 2^{-n}]^d$ due to our cutoff. For all such $y$'s, $\pi_{m,L}(y)$ falls at distance larger than $2^{-n}$ from $x$ except if $m=e_1$. Consequently, $P_n^{(a)}(x,\pi_{m,L}(y))$ vanishes except if $m=e_1$. By symmetry we thus have
\begin{align*}
K_n^{(a)}(x,y) &= P_n^{(a)}(x,y) - P_n^{(a)}(x,\pi_{e_1,L}(y))\\
&= \big(P_n^{(a)}(x,y)-P_n^{(a)}(x_P,y)\big) + \big(P_n^{(a)}(x_P,\pi_{e_1,L}(y))-P_n^{(a)}(x,\pi_{e_1,L}(y))\big)\\
&= \int_{z\in [x,x_P]} \big(\partial^{e_1}_1 P_n^{(a)}(z,y) + \partial^{e_1}_1 P_n^{(a)}(z,\pi_{e_1,L}(y))\big) dz\;.
\end{align*}
Since $|x-x_P|=|x|_P$, there exists some function $\varphi \in \ccB^r$ such that the r.h.s.~can be written as $2^{-n} |x|_P \varphi_x^{5\cdot 2^{-n}}(y)$ up to a multiplicative constant which is bounded uniformly over all the parameters.\\
The general case can be dealt with the same arguments: one simply has to take into account more than a single reflection.
\end{proof}

\subsection{Convolution with the Green's function}

The goal of this subsection is to define the abstract convolution operator $\cKa$ as the lift of the operator $f\mapsto K^{(a)}*f$ at the level of modelled distributions. To that end, we distinguish the singular part $K^{(a)}_+$ from the smooth part $K^{(a)}_-$ of the kernel and define the associated operators $\cKa_+$ and $\cKa_-$ separately. Regarding the former, we set for all $x\in (-L,L)^d$
\begin{equation}\label{Eq:cKa+}\begin{split}
\cKa_+ f(x) &:= \cI(f(x)) + \sum_{\zeta \in \cA_{<\gamma}} \sum_{|k| < \zeta + 2} \frac{X^k}{k!} \big\langle \Pi_x \cQ_\zeta f(x) , \partial^k_1 P_+^{(a)}(x,\cdot)\big\rangle\\
&+ \sum_{|k| < \gamma +2} \frac{X^k}{k!} \big\langle \cR f - \Pi_x f(x) , \partial^k_1 P_+^{(a)}(x,\cdot)\big\rangle\\
&+ \sum_{|k| < \gamma +2} \frac{X^k}{k!} \big\langle \cR f , \partial^k_1 Z_+^{(a)}(x,\cdot)\big\rangle\;.
\end{split}\end{equation}
The expression of the operator $\cKa_-$ is much simpler since the associated kernel is not singular:
$$ \cKa_- f(x) := \sum_{|k| < \gamma'} \frac{X^k}{k!} \langle \cR f , \partial^k_1 K_-^{(a)}(x,\cdot)\rangle\;,\quad x\in (-L,L)^d\;.$$
Note that outside $(-L,L)^d$ these two quantities are set to $0$. Then we set $\cKa := \cKa_+ + \cKa_-$.

\begin{theorem}[Abstract convolution]\label{Th:ConvolAbs}
Fix $\gamma> 0$, $p\in [1,\infty]$ and $\sigma \le \gamma$ such that $\sigma > -1(1-1/p)$. Let $\alpha := \min \cA_{<\gamma}$ and set $\gamma'=\gamma+2$ and $\sigma'=(\sigma +2)\wedge 1$. Assume that $\gamma,\sigma, \alpha \notin \Z$. The operator $\cKa$ is continuous from $\cD^{\gamma,\sigma}_{p,\infty}$ into $\cD^{\gamma',\sigma'}_{p,\infty}$ and satisfies $\cR \cKa f = K^{(a)} * \cR f$. Furthermore, we have the bound
\begin{equation}\label{Eq:Convola}
\$ \cKa f\$_{\gamma',\sigma'} \lesssim \$ f \$_{\gamma,\sigma} \$\Pi\$ (1+\$\Gamma\$)\;,
\end{equation}
uniformly over all $a,L \ge 1$. In the case where we are given another admissible model, we have the bound
\begin{equation}\label{Eq:Convola2}
\$ \cKa f, \cKa \bar f\$_{\gamma',\sigma'} \lesssim C(\Pi,\bar\Pi,\Gamma,\bar\Gamma,f,\bar f)\;,
\end{equation}
uniformly over all admissible models $(\Pi,\Gamma)$, $(\bar\Pi,\bar\Gamma)$, all $f$ and $\bar f$ in $\cD^{\gamma,\sigma}_{p,\infty}$ and $\bar{\cD}^{\gamma,\sigma}_{p,\infty}$ and all $a,L\ge 1$. Here the constant $C(\Pi,\bar\Pi,\Gamma,\bar\Gamma,f,\bar f)$ is given by
$$ \$\Pi\$ (1+\$\Gamma\$)\$f;\bar f\$_{\gamma,\sigma} + \big(\$\Pi-\bar\Pi\$(1+\$\bar\Gamma\$) + \$\bar\Pi\$\$\Gamma-\bar\Gamma\$\big) \$\bar f\$_{\gamma,\sigma}\;.$$
\end{theorem}
\begin{remark}[Periodic]
In the periodic case, the statement remains true but is not sufficient for our purpose. Instead, we can prove that for any $\delta > 0$, if we take $\gamma' = \gamma+2-\delta$, then the restriction of the $\cD^{\gamma'}_{p,\infty}$-norm to the polynomial levels of the regularity structure satisfies the bounds \eqref{Eq:Convola} and \eqref{Eq:Convola2} with a prefactor $a^{-\delta/2}$.
\end{remark}

The reader familiar with regularity structures would have probably expected the parameter $\sigma'$ to be defined as $\sigma' = \sigma\wedge\alpha +2$. Recall that in the original version of the convolution theorem, the expression $\sigma\wedge\alpha$ arose in $\sigma'$ since this is the regularity of the distribution $\cR f$. Due to our choice of weights (for levels below $\sigma$, the weight forces the modelled distribution to ``vanish'' on the boundary), $\cR f$ has regularity $\sigma$ near the boundary so that $\sigma\wedge\alpha$ is replaced by $\sigma$ in our case. However, there is a price to pay for our choice of weights: one needs to show that convolution with the kernel $K^{(a)}$ indeed kills contributions below $\sigma'$. Since our kernel vanishes on the boundary, we are able to prove that the contributions at level $0$ vanish on the boundary as well and therefore we impose the further restriction that $\sigma'$ is below $1$.

\begin{proof}
We present the proof in the case where we deal with a single model, the case with two models follows \textit{mutatis mutandis}. The arguments of the proof are much simpler for the convolution with the smooth part of the kernel than for the singular part of the kernel so we only present the details for the latter.\\
As in the original proof~\cite[Th 5.12]{Hairer2014}, the bounds on the terms in the norm corresponding to non-integer levels are immediate consequences of the definition of the operator and of the properties of $\cI$. Therefore, we concentrate on the terms at integer levels, that is, we only bound the terms $\cQ_k \cKa_+ f(x)$ and $\cQ_k \big(\cKa_+ f(x+h)-\Gamma_{x+h,x}\cKa_+ f(x)\big)$ for all $k\in\N$ such that $|k| < \gamma'$, where $\cQ_k$ is the projection on the vector space generated by the symbol $X^k$ in $\cT_{<\gamma}$. For convenience, we define
\begin{equation}\label{Eq:cKan}\begin{split}
\cKa_n f(x) &:= \sum_{\zeta \in \cA_{<\gamma}} \sum_{|k| < \zeta + 2} \frac{X^k}{k!} \big\langle \Pi_x \cQ_\zeta f(x) , \partial^k_1 P_n^{(a)}(x,\cdot)\big\rangle\\
&+ \sum_{|k| < \gamma +2} \frac{X^k}{k!} \big\langle \cR f - \Pi_x f(x) , \partial^k_1 P_n^{(a)}(x,\cdot)\big\rangle+ \sum_{|k| < \gamma +2} \frac{X^k}{k!} \big\langle \cR f , \partial^k_1 Z_n^{(a)}(x,\cdot)\big\rangle\;.
\end{split}\end{equation}
Recall that $n_a$ is the smallest integer such that $2^{-n_a} \le 1/\sqrt a$. An important remark for the sequel is that $\sigma > - 1(1-1/p) \ge -1$ so that $2+\sigma$ is always strictly larger than $\sigma'$.

\smallskip
 
\textit{Local terms.} We argue differently according to the relative values of $|x|_P$, $2^{-n}$ and $2^{-n_a}$. We first consider the case $3\cdot 2^{-n} \ge |x|_P$. Reordering the terms appearing in \eqref{Eq:cKan}, we find
\begin{equation}\label{Eq:DecompoLocal} k!\cQ_k \cKa_n f(x) = \langle \cR f,\partial^k_1 K_n^{(a)}(x,\cdot)\rangle - \sum_{\zeta \le |k|-2} \langle \Pi_x\cQ_\zeta f(x) , \partial^k P_n^{(a)}(x-\cdot)\rangle\;.
\end{equation}
From the scaling properties of the Green's function $P_n$ and the expression of the kernel $K_n$, we deduce that there exists $c>0$ such that the function $\partial^k K_n^{(a)}(x,\cdot)$ can be written as $2^{-n(2-|k|)} \eta_x^{c2^{-n}}$ for some $\eta\in\ccB^r$ up to a uniformly bounded multiplicative constant. For $k\ge 1$, we have $\sigma'-|k| \le 0$ and we deduce that $2^{-n(\sigma'-|k|)} \lesssim |x|_P^{\sigma'-|k|}$ whenever $|x|_P \le 3\cdot 2^{-n}$. Therefore, we find
\begin{align*}
\sum_{n \ge n_a}\Big\|\frac{\big|\langle \cR f,\partial^k K_n^{(a)}(x,\cdot)\rangle\big|}{|x|_P^{\sigma'-|k|}} \un_{\{|x|_P \le 3\cdot 2^{-n}\}} \Big\|_{L^p_x} &\lesssim \sum_{n \ge n_a} \Big\| \sup_{\eta\in\ccB^r}\frac{\big|\langle \cR f,\eta^{c2^{-n}}_x\rangle\big|}{2^{-n\sigma}} \frac{2^{-n(2-|k|+\sigma)}}{|x|_P^{\sigma'-|k|}}\un_{\{|x|_P \le 4\cdot2^{-n}\}}\Big\|_{L^p_x}\\
&\lesssim \sum_{n \ge n_a} 2^{-n(2+\sigma-\sigma')}\Big\| \sup_{\eta\in\ccB^r}\frac{\big|\langle \cR f,\eta^{c2^{-n}}_x\rangle\big|}{2^{-n\sigma}}\un_{\{|x|_P \le 3\cdot 2^{-n}\}}\Big\|_{L^p_x}\\
&\lesssim \$f\$\;,
\end{align*}
thanks to \eqref{Eq:BoundReconsBdry}.\\
For $k=0$, we use the specific behaviour of the Green's function at the boundary. Namely, by Lemma \ref{Lemma:CancelDirichlet} we have
\begin{align*}
\sum_{n \ge n_a}\Big\|\frac{\big|\langle \cR f, K_n^{(a)}(x,\cdot)\rangle\big|}{|x|_P^{\sigma'}} \un_{\{|x|_P \le 3\cdot 2^{-n}\}} \Big\|_{L^p_x} &\lesssim \sum_{n \ge n_a} \Big\| \sup_{\eta\in\ccB^r}\frac{\big|\langle \cR f,\eta^{c2^{-n}}_x\rangle\big|}{2^{-n\sigma}} \frac{2^{-n(1+\sigma)}}{|x|_P^{\sigma'-1}}\un_{\{|x|_P \le 3\cdot 2^{-n}\}}\Big\|_{L^p_x}\\
&\lesssim \sum_{n \ge n_a} 2^{-n(2+\sigma-\sigma')} \Big\| \sup_{\eta\in\ccB^r}\frac{\big|\langle \cR f,\eta^{c2^{-n}}_x\rangle\big|}{2^{-n\sigma}} \un_{\{|x|_P \le 3\cdot 2^{-n}\}}\Big\|_{L^p_x}\;,
\end{align*}
so that, here again we get a bound of order $\$f\$$.\\
We now bound the second term on the right hand side of \eqref{Eq:DecompoLocal}. We use again the inequality $\sigma+2-\sigma' >0$ to get
\begin{align*}
\Big\|\sum_{n\ge n_a} \frac{\big|\langle \Pi_x\cQ_\zeta f(x),\partial^k P_n^{(a)}(x,\cdot)\rangle\big|}{|x|_P^{\sigma'-|k|}} \un_{\{|x|_P \le 3\cdot 2^{-n}\}}\Big\|_{L^p_x}&\lesssim \sum_{n\ge n_a} \Big\| \frac{|f_\zeta|(x)}{|x|_P^{\sigma-\zeta}} \un_{\{|x|_P \le 3\cdot 2^{-n}\}} \frac{2^{-n(2+\zeta-|k|)}}{|x|_P^{\sigma'-|k|-\sigma+\zeta}}\Big\|_{L^p_x}\\
&\lesssim \sum_{n\ge n_a} \Big\| \frac{|f_\zeta|(x)}{|x|_P^{\sigma-\zeta}} \un_{\{|x|_P \le 3\cdot 2^{-n}\}} |x|_P^{2+\sigma-\sigma'} \Big\|_{L^p_x}\\
&\lesssim \$f\$\;,
\end{align*}
as required.

\smallskip

We now consider the case $3\cdot 2^{-n} \le |x|_P$. Notice that in that case $Z^{(a)}_n(x,\cdot)$ vanishes so that we have
\begin{equation}\label{Eq:DecompoLocal2} k!\cQ_k \cKa_n f(x) = \langle \cR f - \Pi_x f(x),\partial^k P_n^{(a)}(x-\cdot)\rangle + \sum_{\zeta > |k|-2} \langle \Pi_x\cQ_\zeta f(x) , \partial^k P_n^{(a)}(x-\cdot)\rangle\;.
\end{equation}
We have
\begin{align*}
&\sum_{n\ge n_a} \Big\|\frac{\langle \cR f-\Pi_x f(x),\partial^k P_n^{(a)}(x-\cdot)\rangle}{|x|_P^{\sigma'-|k|}} \un_{\{|x|_P \ge 3\cdot 2^{-n}\}}\Big\|_{L^p_x}\\
&\lesssim \sum_{n\ge n_a} \Big\|\sup_{\eta\in\ccB^r} \frac{\big|\langle \cR f-\Pi_x f(x),\eta^{2^{-n}}_x\rangle\big|}{2^{-n\gamma} |x|_P^{\sigma-\gamma}}\un_{\{|x|_P \ge 3\cdot 2^{-n}\}} \frac{2^{-n(2+\gamma-|k|)}}{|x|_P^{\sigma'-|k|+\gamma-\sigma}}\Big\|_{L^p_x}\;.
\end{align*}
Observe that $|x|_P^{-\sigma'+|k|-\gamma+\sigma} \lesssim 2^{-n(-\sigma'+|k|-\gamma+\sigma)} \vee 1$ so that there exists $\delta >0$ such that
$$ \lesssim \sum_{n\ge n_a} 2^{-n\delta}\Big\|\sup_{\eta\in\ccB^r} \frac{\big|\langle \cR f-\Pi_x f(x),\eta^{2^{-n}}_x\rangle\big|}{2^{-n\gamma} |x|_P^{\sigma-\gamma}}\un_{\{|x|_P \ge 3\cdot 2^{-n}\}}\Big\|_{L^p_x}\;,$$
which is of order $\$f\$$ as required. Similarly, we have for all $\zeta > |k|-2$
\begin{align*}
&\sum_{n\ge n_a} \Big\|\frac{\langle \Pi_x\cQ_\zeta f(x) , \partial^k P_n^{(a)}(x-\cdot)\rangle}{|x|_P^{\sigma'-|k|}}\un_{\{|x|_P \ge 3\cdot 2^{-n}\}}\Big\|_{L^p_x}\\ &\lesssim \sum_{n\ge n_a} \Big\|\frac{|f_\zeta|(x)}{|x|_P^{\sigma-\zeta}}\un_{\{|x|_P \ge 3\cdot 2^{-n}\}}\frac{2^{-n(2+\zeta-|k|)}}{|x|_P^{\sigma'-|k|-\sigma+\zeta}}\Big\|_{L^p_x}\\
&\lesssim \sum_{n\ge n_a} 2^{-n\delta} \$f\$\;,
\end{align*}
which is bounded by a term of order $\$f\$$.

\smallskip

\textit{Translation terms.} We introduce the notation
$$ K^{(a),k,\gamma'}_{n,x+h,x} := \partial^k_1 K^{(a)}_{n}(x+h,\cdot) - \sum_{\ell:|k+\ell| < \gamma'} \frac{h^\ell}{\ell!} \partial^{k+\ell}_1 K^{(a)}_{n}(x,\cdot)\;,$$
and similarly
$$ P^{(a),k,\gamma'}_{n,x+h,x} := \partial^k P^{(a)}_{n}(x+h-\cdot) - \sum_{\ell:|k+\ell| < \gamma'} \frac{h^\ell}{\ell!} \partial^{k+\ell} P^{(a)}_{n}(x-\cdot)\;.$$
Taylor's formula allows to get the identity
\begin{equation}\label{Eq:Taylor}
K^{(a),k,\gamma'}_{n,x+h,x} = \int_{u\in [0,1]} \sum_{\ell:|k+\ell| = \lceil \gamma' \rceil} \partial^{k+\ell}_1 K^{(a)}_{n}(x+u h,\cdot) \frac{h^{\ell}}{\ell !} \big(\lceil \gamma'\rceil - |k|\big)(1-u)^{\lfloor \gamma'\rfloor - |k|} du\;,
\end{equation} 
and similarly for $P^{(a),k,\gamma'}_{n,x+h,x}$.\\

We start with the case $|x|_P \le 3\cdot 2^{-n}$. We write
\begin{align*}
&k!\cQ_k\big(\cKa_n f(x+h) - \Gamma_{x+h,x} \cKa_n f(x)\big) = \langle \cR f, K^{(a),k,\gamma'}_{n,x+h,x}\rangle - \langle \Pi_x f(x), P^{(a),k,\gamma'}_{n,x+h,x}\rangle\\
&- \sum_{\zeta \le |k|-2} \langle \Pi_{x+h} \cQ_\zeta \big(f(x+h)-\Gamma_{x+h,x} f(x)\big) , \partial^k P_n^{(a)}(x+h-\cdot) \rangle
\end{align*}
We bound separately the three terms on the r.h.s. The first term is dealt with using \eqref{Eq:Taylor}. Regarding the second term, we use the same identity to get
\begin{align*}
&\sum_{n\ge n_a}\sup_{h: |h| < 1/\sqrt a} \Big\| \frac{\langle \Pi_x f(x), P^{(a),k,\gamma'}_{n,x+h,x}\rangle}{|h|^{\gamma'-|k|} |x|_P^{\sigma'-\gamma'}} \un_{\{3|h| \le |x|_P \le 3 \cdot 2^{-n}\}} \Big\|_{L^p_x} \\
&\lesssim \sum_{n\ge n_a} \sum_{\ell:|k+\ell| =\lceil \gamma' \rceil} \sup_{h: |h| < 1/\sqrt a} \Big\| \frac{|f_\zeta|(x)}{|x|_P^{\sigma-\zeta}} 2^{-n(2-|k+\ell|+\zeta)} |h|^{|k+\ell|-\gamma'} |x|_P^{\sigma-\zeta-\sigma'+\gamma'}\un_{\{3|h| \le |x|_P \le 3\cdot 2^{-n}\}} \Big\|_{L^p_x} \;.
\end{align*}
At this point, we observe that $|k+\ell|-\gamma' > 0$, $|k+\ell| > 2+\gamma$ and $\zeta < \gamma$ so that we obtain the further bound
\begin{align*}
&\lesssim \sum_{n\ge n_a} \sum_{\ell:|k+\ell| =\lceil \gamma' \rceil} \Big\| \frac{|f_\zeta|(x)}{|x|^{\sigma-\zeta}} 2^{-n(2+\sigma-\sigma')}\un_{\{|x|_P \le 3 \cdot 2^{-n}\}} \Big\|_{L^p_x}\lesssim \$f\$\;.
\end{align*}
The bound of the third term is similar.\\

We turn to the case $3|h| \le 3\cdot 2^{-n} \le |x|_P$. Notice that in that case, $K^{(a)}_n(x,\cdot) = P^{(a)}_n(x,\cdot)$ so that we can write
\begin{align*}
&k!\cQ_k\big(\cKa_n f(x+h) - \Gamma_{x+h,x} \cKa_n f(x)\big) = \langle \cR f-\Pi_x f(x), P^{(a),k,\gamma'}_{n,x+h,x}\rangle\\
&- \sum_{\zeta \le |k|-2} \langle \Pi_{x+h} \cQ_\zeta \big(f(x+h)-\Gamma_{x+h,x} f(x)\big) , \partial^k P_n^{(a)}(x+h-\cdot) \rangle\;.
\end{align*}
Using the reconstruction bound \eqref{Eq:ReconsBound}, we obtain
\begin{align*}
&\sup_{h: |h| < 1/\sqrt a} \Big\| \sum_{n\ge n_a}\frac{\langle \cR f - \Pi_x f(x), P^{(a),k,\gamma'}_{n,x+h,x}\rangle}{|h|^{\gamma'-|k|} |x|_P^{\sigma'-\gamma'}} \un_{\{3|h| \le 3\cdot 2^{-n} \le |x|_P\}} \Big\|_{L^p_x} \\
&\lesssim \sum_{\ell:|k+\ell| =\lceil \gamma' \rceil}\sup_{h: |h| < 1/\sqrt a}\sum_{n\ge n_a} (|h|2^n)^{|k+\ell|-\gamma'}\un_{\{|h| \le 2^{-n}\}}\\
&\qquad\qquad\qquad\times\Big\| \sup_{\eta \in \ccB^r} \frac{\big|\langle \cR f - \Pi_x f(x), \eta^{2\cdot 2^{-n}}_x\rangle\big|}{2^{-n\gamma} |x|_P^{\sigma-\gamma}} 2^{-n(2+\gamma-\gamma')}|x|_P^{\sigma-\gamma-\sigma'+\gamma'}  \un_{\{3\cdot 2^{-n} \le |x|_P\}} \Big\|_{L^p_x}\\
&\lesssim \sup_{n\ge n_a} \Big\| \sup_{\eta \in \ccB^r} \frac{\big|\langle \cR f - \Pi_x f(x), \eta^{2\cdot 2^{-n}}_x\rangle\big|}{2^{-n\gamma} |x|_P^{\sigma-\gamma}} \un_{\{3\cdot 2^{-n} \le |x|_P\}} \Big\|_{L^p_x}\;,
\end{align*}
as required. The bound of the second term is simpler so we do not present the details.\\
We now consider the case $3\cdot 2^{-n} \le 3|h| \le |x|_P$. We write
\begin{align*}
&k!\cQ_k\big(\cKa_n f(x+h) - \Gamma_{x+h,x} \cKa_n f(x)\big) = \langle \cR f-\Pi_{x+h} f(x+h), \partial^k P^{(a)}_{n}(x+h-\cdot)\rangle\\
&-\langle \cR f-\Pi_{x} f(x), \sum_{\ell:|k+\ell|<\gamma'} \frac{h^\ell}{\ell!}\partial^{k+\ell} P^{(a)}_{n}(x-\cdot)\rangle\\
&+ \sum_{\zeta > |k|-2} \langle \Pi_{x+h} \cQ_\zeta \big(f(x+h)-\Gamma_{x+h,x} f(x)\big) , \partial^k P_n^{(a)}(x+h-\cdot) \rangle\;.
\end{align*}
Let us present in detail the bound of the first term. We have
\begin{align*}
&\sup_{h: |h| < 1/\sqrt a} \Big\| \sum_{n\ge n_a}\frac{\langle \cR f-\Pi_{x+h} f(x+h), \partial^k P^{(a)}_{n}(x+h-\cdot)\rangle}{|h|^{\gamma'-|k|} |x|_P^{\sigma'-\gamma'}} \un_{\{3\cdot 2^{-n} \le 3|h| \le |x|_P\}} \Big\|_{L^p_x}\\
&\lesssim\sup_{h: |h| < 1/\sqrt a} \Big\| \sum_{n\ge n_a}\sup_{\eta\in\ccB^r}\frac{\big|\langle \cR f-\Pi_{y} f(y), \eta^{2^{-n}}_y\rangle\big|}{2^{-n\gamma} |y|_P^{\sigma-\gamma}} \un_{\{2\cdot 2^{-n} \le 2|h| \le |y|_P\}}2^{-n(2-|k|+\gamma)}|h|^{|k|-\gamma'}|y|_P^{\sigma-\gamma+\gamma'-\sigma'} \Big\|_{L^p_y}\\
&\lesssim \sup_{h: |h| < 1/\sqrt a}\sum_{n\ge n_a} \Big(\frac{2^{-n}}{|h|}\Big)^{\gamma'-|k|}\un_{\{2^{-n} \le |h| \}} \Big\| \sup_{\eta\in\ccB^r}\frac{\big|\langle \cR f-\Pi_{y} f(y), \eta^{2^{-n}}_y\rangle\big|}{2^{-n\gamma} |y|_P^{\sigma-\gamma}} \un_{\{2\cdot 2^{-n}\le |y|_P\}}\Big\|_{L^p_y}\\
&\lesssim \sup_{n\ge n_a} \Big\| \sup_{\eta\in\ccB^r}\frac{\big|\langle \cR f-\Pi_{y} f(y), \eta^{2^{-n}}_y\rangle\big|}{2^{-n\gamma} |y|_P^{\sigma-\gamma}} \un_{\{2\cdot 2^{-n} \le |y|_P\}} \Big\|_{L^p_y}\;,
\end{align*}
which is bounded by a term of order $\$f\$$.

\smallskip

\textit{Convolution identity.} We already know that $\cK^{(a)}_+ f \in \cD^{\gamma',\sigma'}_{p,\infty}$. Up to taking $\gamma''$ smaller than $\gamma'$ and $\sigma'' = \sigma'\wedge \gamma''$, we can always assume that $\gamma'\in (0,1)$. The uniqueness part of the reconstruction theorem away from the boundary $P$, together with the fact that $K^{(a)}_+*\cR f$ is completely determined by its evaluations away from $P$, ensure that, in order to show that $\cR \cK^{(a)}_+ f = K^{(a)}_+ * \cR f$, it suffices to establish the following bound:
\begin{equation*}
\sup_{\lambda \in (0,1/\sqrt a)} \Big\| \sup_{\eta\in\ccB^r} \frac{\big|\langle K^{(a)}_+ * \cR f - \Pi_x \cK^{(a)}_+ f(x) , \eta^\lambda_x\rangle\big|}{\lambda^{\gamma'} |x|_P^{\sigma'-\gamma'}} \un_{\{3\lambda < |x|_P\}}\Big\|_{L^p_x} < \infty\;.
\end{equation*}
Using the fact that $P^{(a)}_+$, and therefore $K^{(a)}_+$, annihilates all polynomials of degree at most $r$, a straightforward computation shows the following identity:
\begin{align*}
\langle K^{(a)}_+ * \cR f - \Pi_x \cK^{(a)}_+ f(x) , \eta^\lambda_x\rangle = \int_h \sum_{n\ge n_a} \Big(\langle \cR f-\Pi_x f(x), P^{(a),0,\gamma'}_{n,x+h,x}\rangle + \langle \cR f , Z^{(a),0,\gamma'}_{n,x+h,x}\rangle\Big) \eta^\lambda(h) dh\;.
\end{align*}
Treating separately the terms $n$ such that $3\lambda \le |x|_P \le 3\cdot 2^{-n}$, $3\lambda \le 3\cdot 2^{-n} \le |x|_P$ and $3\cdot 2^{-n} \le 3 \lambda \le |x|_P$ and applying the arguments presented in the \textit{Translation terms} bounds, it is straightforward to get the required estimate.
\end{proof}

The next result shows how to lift in the regularity structure the convolution of some classical distribution with $K^{(a)}$.

\begin{proposition}\label{Prop:ConvolClassical}
For $\delta \in (0,1/2)$, let $g \in \cB^{-\delta}_{p,\infty}$, take $\gamma = 2 - \delta$ and $\sigma = 1$. Then, the modelled distribution:
$$ \cK^{(a)} g(x) :=\sum_{k\in\N^d: |k| < \gamma} \frac{X^k}{k!} \langle g, \partial^k_1 K^{(a)}(x,\cdot) \rangle\;,\quad x\in (-L,L)^d\;,$$
belongs to the space $\cD^{\gamma,\sigma}_{p,\infty}$ and satisfies the following bound uniformly over all $a,L\ge 1$ and all $g\in\cB^{-\delta}_{p,\infty}$
$$ \$ \cK^{(a)} g\$_{\gamma,\sigma} \lesssim \| g\|_{\cB^{-\delta}_{p,\infty}}\;.$$
\end{proposition}
\begin{proof}
Let us consider the terms in $\cK^{(a)} g(x)$ arising from the singular part $K^{(a)}_+$ of the kernel: for convenience, let $\cK^{(a)}_+ g(x)$ be the same expression as above except that $K^{(a)}$ is replaced by $K^{(a)}_+$.\\
We start with the local bounds of the $\cD^{\gamma,\sigma}_{p,\infty}$-norm. We aim at bounding
$$ \sum_{n\ge n_a} \Big\| \frac{\langle g, \partial^k_1 K^{(a)}_n(x,\cdot)\rangle}{|x|_P^{\sigma-|k|}} \Big\|_{L^p_x}\;,$$
for any $k\in\N^d$ such that $|k| < \gamma$. We first consider the case $2^{-n} \le |x|_P$. We have
$$ \sum_{n\ge n_a} \Big\| \frac{\langle g, \partial^k_1 K^{(a)}_n(x,\cdot)\rangle}{|x|_P^{\sigma-|k|}} \un_{\{|x|_P \ge 2^{-n}\}} \Big\|_{L^p_x} \lesssim \begin{cases} \sum_{n\ge n_a} 2^{-n(2-|k|-\delta)}  \| g\|_{\cB^{-\delta}_{p,\infty}}\quad &\mbox{ if } |k| = 1\;,\\
\sum_{n\ge n_a} 2^{-n(2-\delta-\sigma)} \| g\|_{\cB^{-\delta}_{p,\infty}}\quad &\mbox{ if } k=0\;.
\end{cases}
$$
In any case, this yields a term of order $\| g\|_{\cB^{-\delta}_{p,\infty}}$. We now consider the case $2^{-n} >|x|_P$. When $|k| = 1$, we have
$$ \sum_{n\ge n_a} \Big\| \frac{\langle g, \partial^k_1 K^{(a)}_n(x,\cdot)\rangle}{|x|_P^{\sigma-|k|}} \un_{\{|x|_P \le 2^{-n}\}} \Big\|_{L^p_x} \lesssim \sum_{n\ge n_a} 2^{-n(2-|k|-\delta)} \| g\|_{\cB^{-\delta}_{p,\infty}} \lesssim \| g\|_{\cB^{-\delta}_{p,\infty}}\;.$$
When $k=0$, we rely on Lemma \ref{Lemma:CancelDirichlet} to get
\begin{align*}
\sum_{n\ge n_a} \Big\| \frac{\langle g, \partial^k_1 K^{(a)}_n(x,\cdot)\rangle}{|x|_P^{\sigma-|k|}} \un_{\{|x|_P \le 2^{-n}\}} \Big\|_{L^p_x} &\lesssim \sum_{n\ge n_a} \Big\| \sup_{\eta\in\ccB^r} \frac{\langle g, \eta^{c2^{-n}}_x \rangle}{2^{n\delta}} |x|_P^{1-\sigma} 2^{-n(1-\delta)} \un_{\{|x|_P \le 2^{-n}\}} \Big\|_{L^p_x}\\
&\lesssim \sum_{n\ge n_a} 2^{-n(1-\delta)} \| g\|_{\cB^{-\delta}_{p,\infty}} \lesssim \| g\|_{\cB^{-\delta}_{p,\infty}}\;.
\end{align*}
To obtain the translation bounds of the $\cD^{\gamma,\sigma}_{p,\infty}$-norm, one needs to distinguish three cases according to the relative values of $2^{-n}$, $|h|$ and $|x|_P$ and then to argue similarly as in the previous proof.\\
To treat the convolution with the smooth part $K^{(a)}_-$ of the kernel, the arguments are similar so we do not present the details.
\end{proof}

\subsection{Fixed point results}\label{Subsec:Fixed}

Recall that $\kappa > 0$ can be taken as small as desired. We set
$$ p = \begin{cases} 2 &\mbox{ if } d\le 2\;,\\
\frac{1}{2-\frac{d}{2}-3\kappa}&\mbox{ if } d=3\;,
\end{cases}\qquad\gamma = 2 - d \Big(\frac12 - \frac1{p}\Big)-\kappa\;,\qquad \sigma = 1-\kappa\;.$$
This choice of parameters ensures that the space $\cD^{\gamma-\frac{d}{2}-\kappa,\sigma-\frac{d}{2}-\kappa}_{p,\infty}$ (in which $f\cdot\Xi$ lives whenever $f\in \cD^{\gamma,\sigma}_{p,\infty}$) falls into the scope of Theorem \ref{Th:Recons}. Note that, unfortunately, in dimension $3$ the natural choice $p=2$ leads to
$$ \sigma - \frac{d}{2} - \kappa = -\frac12 - 2\kappa < -1(1-\frac1{2})\;,$$
so that one needs to increase slightly $p$ as we did above for the reconstruction theorem to apply.

\begin{proposition}\label{Prop:Resolv}
There exists a positive function $a\mapsto A(a)$ increasing to infinity and independent of $L\ge 1$ such that the following holds. For any admissible model $(\Pi^{(a)},\Gamma^{(a)})$ such that $\$\Pi^{(a)}\$(1+ \$\Gamma^{(a)}\$) < A(a)$, for any $g\in L^2$ and any $b\in (-2,2)$, the map
$$ \cM_a: f \mapsto \cK^{(a)} g - \cK^{(a)}(f\cdot\Xi) - b\,\cK^{(a)} f\;,$$
admits a unique fixed point in $\cD^{\gamma,\sigma}_{p,\infty}$ that we denote by $\cG^{a,b} g$. Furthermore, there exists a constant $C>0$ such that for any two admissible models $(\Pi^{(a)},\Gamma^{(a)})$ and $(\bar\Pi^{(a)},\bar\Gamma^{(a)})$ that satisfy $\$\Pi^{(a)}\$(1+ \$\Gamma^{(a)})\$ < A(a)$ and $\$\bar\Pi^{(a)}\$(1+ \$\bar\Gamma^{(a)})\$ < A(a)$, we have uniformly over all $g,\bar g \in L^2$, all $b,\bar b\in (-2,2)$ and all $a,L\ge 1$
\begin{equation}\label{Eq:ContBnd}
\$ \cG^{a,b} g ; \bar\cG^{a,b} \bar g \$ \le C \Big( \|g-\bar g\|_{L^2} + \|\bar g\|_{L^2} \big(|b-\bar b|+ \$ \Pi^{(a)} - \bar\Pi^{(a)} \$ (1+\$\bar\Gamma^{(a)}\$) + \$ \bar\Pi^{(a)}\$ \$\Gamma^{(a)}-\bar\Gamma^{(a)}\$ \big) \Big) \;.
\end{equation}
\end{proposition}
\begin{remark}
We consider the additional parameter $b$ in order to construct all the resolvent operators associated to $a+b$, $b\in [-2,2]$, with the sole model constructed with the kernel $P^{(a)}_+$. We refer to Subsection \ref{Subsection:Resolv} for more details.
\end{remark}
An important observation, which we will use several times in the proof, is that the embedding of $\cD^{\gamma+\kappa,\sigma+\kappa}_{p,\infty}$ into $\cD^{\gamma,\sigma}_{p,\infty}$ has a norm of order $a^{-\kappa/2}$ uniformly over all $a\ge 1$.
\begin{proof}
Let $g\in L^2$. From classical embedding theorems, $g\in \cB^{-d(\frac12 - \frac1{p})}_{p,\infty}$. By Proposition \ref{Prop:ConvolClassical} we know that $\cK^{(a)} g \in \cD^{\gamma,\sigma}_{p,\infty}$ and the following bound holds true
$$ \$ \cK^{(a)} g\$_{\gamma,\sigma} \lesssim a^{-\kappa/2}  \| g \|_{L^2}\;.$$
If $f\in \cD^{\gamma,\sigma}_{p,\infty}$, it is elementary to check that $f\cdot\Xi \in \cD^{\gamma-\frac{d}{2}-\kappa,\sigma-\frac{d}{2}-\kappa}_{p,\infty}$. Applying Theorem \ref{Th:ConvolAbs} we also have $\cK^{(a)} (f\cdot\Xi) \in \cD^{\gamma,\sigma}_{p,\infty}$ and furthermore
$$ \$ \cKa(f\cdot\Xi)\$_{\gamma,\sigma} \lesssim a^{-\kappa/2}\$ f \$_{\gamma,\sigma} \$\Pi\$ (1+\$\Gamma\$)\;.$$
Similarly, $\cK^{(a)} f \in \cD^{\gamma,\sigma}_{p,\infty}$ and
$$ \$ \cKa f\$_{\gamma,\sigma} \lesssim a^{-\kappa/2} \$ f \$_{\gamma,\sigma} \$\Pi\$ (1+\$\Gamma\$)\;.$$
Using the linearity of the map $\cM_a$, we deduce that there exists a constant $C>0$ independent of everything such that
$$ \$\cM_a f - \cM_a \bar f\$_{\gamma,\sigma} \le C a^{-\kappa/2} \$\Pi\$ (1+\$\Gamma\$) \$ f-\bar f\$_{\gamma,\sigma}\;,$$
uniformly over all $f,\bar f \in \cD^{\gamma,\sigma}_{p,\infty}$ and over all $a,L \ge 1$. Choosing $A(a)$ small enough, we deduce that $\cM_a$ is a contraction on $\cD^{\gamma,\sigma}_{p,\infty}$ so that it admits a unique fixed point $f$.\\
In the case where we are given two models, we let $f$ and $\bar f$ be the corresponding fixed points. Since $\cM^a f=f$ and $\bar\cM^a \bar f = \bar f$ we have
$$ \$ f ;  \bar f\$ \le \$\cK^{(a)} (g-\bar g)\$ + \$\cK^{(a)}(f\cdot\Xi) ; \cK^{(a)}(\bar f\cdot\Xi) \$ +   b \$ \cK^{(a)} f ; \cK^{(a)} \bar f \$ + |b-\bar b|\$\cK^{(a)} \bar f \$ \;,$$
so that, using the bound of Theorem \ref{Th:ConvolAbs} in the case of two models, we deduce that choosing $A(a)$ small enough the asserted result holds true.
\end{proof}
\begin{remark}[Periodic]
In the periodic case, the proof is substantially different. Indeed, Theorem \ref{Th:ConvolAbs} only gives contractivity at integer levels in the regularity structure. However, a careful inspection of the relationship between the coefficients of $\cM_a f$ and $f$ shows that if one iterates $k$ times the map $\cM_a$ (with $k=2$ in dimension $1$ and $2$, and $k=4$ in dimension $3$) then the coefficients at non-integer levels of $\cM_a^k f$ coincide with coefficients at integer levels of $\cM_a f$, and we get contractivity.
\end{remark}
Let us observe that the fixed point of the last proposition takes the generic form
$$ f(x) =\sum_{\tau \in \cU_{<\gamma}} f_\tau(x) \,\tau\;,\quad x\in (-L,L)^d\;,$$
and that the coefficients $f_\tau$ for non-polynomials $\tau$ take values in $\{\pm f_\un, \pm f_{X_i}\}$. In particular, in dimension $3$ we have
$$ -f_{\cI(\Xi)} = f_{\cI(\Xi\cI(\Xi))} = - f_{\cI(\Xi\cI(\Xi\cI(\Xi)))} = f_{\un}\;,\quad -f_{\cI(X_i \Xi)}  = f_{X_i}\;.$$

\section{The Anderson hamiltonian}\label{Sec:Anderson}

In this section, we apply the previous analytical results to a specific sequence of models based on white noise called the renormalised model: its construction is recalled in the first subsection.

\subsection{The renormalised model}\label{Sec:Models}

Let $\xi$ be a white noise on $\R^d$ (in the periodic case: one imposes this white noise to have the periodicity of the underlying domain). Fix some smooth, compactly supported, even function $\varrho$ integrating to $1$, and set $\varrho_\epsilon(x) := \epsilon^{-d} \varrho(x \epsilon^{-1})$ for every $x\in \R^d$. Then we consider the smooth function $\xi_\epsilon := \xi * \varrho_\epsilon$ from which we can build a model $(\Pi_\epsilon^{(a)},\Gamma_\epsilon^{(a)})$ by setting
$$ (\Pi_\epsilon^{(a)})_x \Xi(y) := \xi_\epsilon(y)\;,\quad x,y \in \R^d\;,$$
by imposing the last two identities of Definition \ref{Def:AdmModel} with the kernel $P_+^{(a)}$ as well as the recursive identities
$$(\Pi_\epsilon^{(a)})_x \tau\tau'(y) = (\Pi_\epsilon^{(a)})_x \tau(y)(\Pi_\epsilon^{(a)})_x \tau'(y)\;,\quad \tau,\tau' \in \cT_{<\gamma}\;.$$
From~\cite[Prop 8.27]{Hairer2014}, this defines a unique admissible model $(\Pi_\epsilon^{(a)},\Gamma_\epsilon^{(a)})$.

Unfortunately in dimension $2$ and $3$, the sequence does not converge as $\epsilon\downarrow 0$ and we need to renormalise the model before passing to the limit. The renormalisation constants were computed in~\cite{Hairer2014} for $d=2$, and in~\cite{HaiPar,mSHE} for $d=3$: their expressions are exactly the same here except that the kernel is $P_+^{(a)}$. For $d=2$, we take $C_\epsilon^{(a)} = c_\epsilon^{(a)}$ with
$$ c_\epsilon^{(a)} := \int P^{(a)}_+(x) \varrho_\epsilon^{*2}(x) \,dx\;,$$
where $\varrho_\epsilon^{*2} := \varrho_\epsilon*\varrho_\epsilon$. For any given $a\ge 1$, we have $C_\epsilon^{(a)} = -(2\pi)^{-1} \ln \epsilon + \cO(1)$ as $\epsilon\downarrow 0$.\\
For $d=3$, we take $C_\epsilon^{(a)} := c^{(a)}_\epsilon + c^{(a),1,1}_\epsilon + c^{(a),1,2}_\epsilon$ where
\begin{align*}
c^{(a)}_\epsilon &:= \int P^{(a)}_+(x) \varrho_\epsilon^{*2}(x) \,dx\;,\\
c^{(a),1,1}_\epsilon &:= \iiint P^{(a)}_+(x_1) P^{(a)}_+(x_2) P^{(a)}_+(x_3) \varrho_\epsilon^{*2}(x_1+x_2)\varrho_\epsilon^{*2}(x_2+x_3) \,dx_1 \,dx_2 \,dx_3\;,\\
c^{(a),1,2}_\epsilon &:= \iiint P^{(a)}_+(x_1) P^{(a)}_+(x_2) \big(P^{(a)}_+(x_3) \varrho_\epsilon^{*2}(x_3) - c_\epsilon^a \delta_0(x_3) \big)\varrho_\epsilon^{*2}(x_1+x_2+x_3) \,dx_1 \,dx_2 \,dx_3\;.
\end{align*}
Note that there exist some constants $c_\varrho, c^{1,1}_\varrho$ independent of $a$ such that for any given $a\ge 1$ as $\epsilon\downarrow 0$
$$ C^{(a)}_\epsilon = \frac{c_\varrho}{\epsilon} + c^{1,1}_\varrho \ln \epsilon + \cO(1)\;.$$
For $d=2,3$ and as $\epsilon \downarrow 0$, $C^{(a)}_\epsilon - C^{(1)}_\epsilon$ converges to a finite quantity that we denote $C^{(a)-(1)}$. Furthermore
\begin{equation}\label{Eq:BoundCaa'}
\sup_{\epsilon \in (0,1)} \big|C^{(a)}_\epsilon - C^{(a')}_\epsilon\big| \lesssim |\sqrt a - \sqrt{a'}|\;,\qquad \big|C^{(a) - (1)} - C^{(a') - (1)}\big| \lesssim |\sqrt a - \sqrt{a'}|\;,
\end{equation}
uniformly over all $a,a'\ge 1$.

The precise definition of the renormalised model requires to introduce several algrebraic objects related to the structure group: we refer the interested reader to~\cite[Section 9.1]{Hairer2014} for the case of dimension $2$ and to~\cite[Section 4.3]{HaiPar} for the case of dimension $3$ (note that in the latter reference, the SPDE at stake is actually the multiplicative SHE whose scaling behaviour is equivalent to that of our operator in dimension $3$). Let us mention that for any $x\in\R^d$ we have
$$ (\Pi_\epsilon^{(a)})_x \Xi \cI(\Xi) = -c_\epsilon^{(a)}\;,$$
in dimension $2$ and $3$, and
$$  (\Pi_\epsilon^{(a)})_x \Xi \cI(\Xi\cI(\Xi\cI(\Xi))) = -c_\epsilon^{(a),1,1} - c_\epsilon^{(a),1,2}\;,$$
in dimension $3$.\\
We conclude this subsection with a convergence result.
\begin{proposition}\label{Prop:Renorm}
For any $a\ge 1$, the sequence of renormalised models $(\Pi^{(a)}_\epsilon,\Gamma^{(a)}_\epsilon)$ converges in probability as $\epsilon\downarrow 0$ to a limit $(\Pi^{(a)},\Gamma^{(a)})$. Furthermore, there exist two constants $K,C>0$ such that for any $a,L\ge 1$ we have
\begin{equation}\label{Eq:BoundModel}
\P(\$ \Pi^{(a)} \$ (1+ \$ \Gamma^{(a)} \$) > K) \le \frac{C}{a^2}\;.
\end{equation}
\end{proposition}
\begin{proof}
The convergence of the sequence of models for any given $a\ge 1$ is already proved in~\cite{Hairer2014,HaiPar}: indeed, the only requirement therein - translated into our context - is that the kernel $P^{(a)}_+$ coincides in a neighbourhood of the origin with the Green's function of $-\Delta+a$ and this is the case in our setting.\\
To establish \eqref{Eq:BoundModel}, we first observe that by Lemma \ref{Lemma:ModelLattice} it suffices to control the tail of the distribution of $\$ \Pi^{(a)} \Latt$. The norm of symbols containing at least one instance of $\Xi$ are bounded by the forthcoming Lemma \ref{Lemma:Moments} from which it is then simple to get \eqref{Eq:BoundModel}.
\end{proof}

\subsection{The resolvents}\label{Subsection:Resolv}

In this subsection, we deal with the collection of renormalised models $(\Pi^{(a)}_\epsilon,\Gamma^{(a)}_\epsilon)$, $\epsilon \in (0,1)$, and the limiting model $(\Pi^{(a)},\Gamma^{(a)})$ introduced in the previous subsection. Since convergence along subsequences of $\epsilon$ of the renormalised models towards the limiting model holds up to a $\P$-null set depending on $a$, we cannot deal simultaneously with all models indexed by $a \ge 1$. Instead, we restrict ourselves to models indexed by $m\in\N$.\\

We introduce the random sets
$$ \ccA := \Big\{ m+b+C^{(m) - (1)}: m\ge 1, b\in (-2,2) \mbox{ s.t. }\$ \Pi^{(m)} \$ (1+ \$ \Gamma^{(m)} \$) < \frac12 A(m)\Big\}\;,$$
and
$$ \ccA_\epsilon := \Big\{ m+b+C^{(m)}_\epsilon - C^{(1)}_\epsilon: m\ge 1, b\in (-2,2) \mbox{ s.t. }\$ \Pi^{(m)}_\epsilon \$ (1+ \$ \Gamma^{(m)}_\epsilon \$) < A(m)\Big\}\;.$$
Recall from Proposition \ref{Prop:Resolv} and \eqref{Eq:BoundCaa'} that $A(m) \to \infty$ and $C^{(m) - (1)} = o(m)$ as $m\to\infty$. Combining the bound \eqref{Eq:BoundModel} with the Borel-Cantelli Lemma, we deduce that $\ccA$ is almost surely non empty, bounded from the left and unbounded from the right.\\

For every $a\in \ccA$ (resp. $a\in \ccA_\epsilon$), we apply the fixed point result of Subsection \ref{Subsec:Fixed} and define the operators
$$ G^a g := \cR \cG^{m,b} g\;,\quad G^a_{\epsilon} g := \cR_\epsilon \cG^{m,b}_{\epsilon} g\;,\quad g\in L^2\;,$$
where $(m,b)$ is an arbitrary pair such that $a=m+b+C^{(m) - (1)}$ (resp. $a=m+b+C^{(m)}_\epsilon - C^{(1)}_\epsilon$). We will show below that this definition does not depend on the chosen pair.

\begin{proposition}\label{Prop:Compact}
Almost surely, for every $a\in \ccA$ (resp. $a\in \ccA_\epsilon$) the operator $G^{a}$ (resp. $G^a_{\epsilon}$) is invertible, compact and self-adjoint. Furthermore the following resolvent identity holds
$$ G^{a} - G^{a'} = (a'-a) G^{a'} G^{a}\;,\quad G^{a}_{\epsilon} - G^{a'}_{\epsilon} = (a'-a) G^{a'}_{\epsilon} G^{a}_{\epsilon}\;.$$
\end{proposition}

The proof relies on two intermediate lemmas. For every $m\in\N$, we introduce the following events
$$ E_m = \Big\{\omega\in\Omega: \$ \Pi^{(m)} \$ (1+ \$ \Gamma^{(m)} \$) < \frac12 A(m) \Big\}\;,\quad E_{m,\epsilon} = \Big\{\omega\in\Omega: \$ \Pi^{(m)}_\epsilon \$ (1+ \$ \Gamma^{(m)}_\epsilon \$) <  A(m) \Big\}\;.$$
On the event $E_m$ and for every $b\in (-2,2)$, we introduce the operator
$$ G^{m,b} g := \cR \cG^{m,b} g\;,\quad g\in L^2\;,$$
and similarly with the regularised model.

\begin{lemma}\label{Lemma:Compact}
On the event $E_m$ (resp. $E_{m,\epsilon}$) and for every $b\in (-2,2)$, the operator $G^{m,b}$ (resp. $G^{m,b}_\epsilon$) is invertible and compact.
\end{lemma}
\begin{proof}
The proof is the same for $G^{m,b}$ and $G^{m,b}_\epsilon$, so we only consider the first operator.
For all $g\in L^2$, $G^{m,b} g = \cR \cG^{m,b} g \in \cB^{2-\frac{d}{2}-\kappa}_{p,\infty}$. In dimension $d=3$, $p$ is larger than $2$ but classical embedding theorems yield that the latter space is continuously embedded into $\cB^{1/2-2\kappa}_{2,2}$. In any cases, $G^{m,b}$ takes values in a compact subspace of $L^2$ so we deduce that $G^{m,b}$ is a compact operator from $L^2$ into itself.

We now prove injectivity of $G^{m,b}$. Let $g\in L^2$ be such that $G^{m,b} g = 0$. This implies that $\langle \cG^{m,b} g, \un\rangle =0$. The fixed point identity satisfied by $\cG^{m,b} g$ implies that its coefficients at non-integer levels equal those at integer levels (up to changes of signs), see the discussion below Proposition \ref{Prop:Resolv}. In particular, if we take $\gamma' = 1+\kappa$, then the restriction of $\cG^{m,b} g$ to $\cT_{<\gamma'}$ has non-zero contributions only at the $X_i$'s:
$$\cQ_{<\gamma'} \cG^{m,b} g = \sum_{i=1}^d \langle \cG^{m,b} g, X_i\rangle X_i\;.$$
Since the reconstruction operator is a bijection between $\cD^{\gamma'}_{p,\infty}(\bar{\cT})$ (the restriction of $\cD^{\gamma'}_{p,\infty}$ to the polynomial regularity structure) and $\cB^{\gamma'}_{p,\infty}$, see~\cite[Prop. 3.4]{Recons}, and since away from the hyperplane $P$ the space $\cD^{\gamma',\sigma}_{p,\infty}$ is locally identical to $\cD^{\gamma'}_{p,\infty}$ we deduce that $\cQ_{<\gamma'} \cG^{m,b} g = 0$. Using again the relationship between the projections of $\cG^{m,b} g$ on integer and non-integer levels, we deduce that $\cG^{m,b} g = 0$. Plugging this into the fixed point equation
$$ \cG^{m,b} g = \cK^{(m)} g - \cK^{(m)}(\cG^{m,b} g\cdot\Xi) - b\,\cK^{(m)}(\cG^{m,b} g)\;,$$
we deduce that $\cK^{(m)} g = 0$ which in turn ensures that $g=0$, thus concluding the proof of the injectivity of $G^{m,b}$. Consequently, $G^{m,b}$ is invertible from $L^2$ into its range.
\end{proof}
\begin{remark}\label{Remark:Embedding}
In this proof, we relied on the continuous embedding of $L^p$ into $L^2$. Note that the norm of this embedding grows with $L$ like $L^{d(1-\frac{2}{p})}$.
\end{remark}
\begin{lemma}\label{Lemma:ResolvId}
On the event $E_{m,\epsilon}$ and for every $b\in (-2,2)$, the operator $G^{m,b}_{\epsilon}$ is self-adjoint and its inverse is given by
$$ -\Delta + m + b + \xi_{\epsilon} + C_{\epsilon}^{(m)}\;.$$
Furthermore, on the event $E_{m,\epsilon} \cap E_{m',\epsilon}$ it satisfies the following resolvent identity
\begin{equation}\label{Eq:ResolvEps}
G^{m,b}_{\epsilon} - G^{m',b'}_{\epsilon} = (m'+b'-m-b + C^{(m')}_\epsilon - C^{(1)}_\epsilon - C^{(m)}_\epsilon + C^{(1)}_\epsilon) G^{m',b'}_{\epsilon} G^{m,b}_{\epsilon}\;.
\end{equation}
\end{lemma}
The second part of the lemma ensures that $G^{m,b}_{\epsilon} = G^{m',b'}_{\epsilon}$ as soon as $m+b+C^{(m)}_\epsilon - C^{(1)}_\epsilon = m'+b'+ C^{(m')}_\epsilon - C^{(1)}_\epsilon$.
\begin{proof}
Let us start with self-adjointness. By Theorem \ref{Th:ConvolAbs}, we have on the event $E_{m,\epsilon}$
$$ G^{m,b}_\epsilon g = K^{(m)} * g - K^{(m)}* \cR_\epsilon(\cG^{m,b}_\epsilon g \cdot \Xi) - b\, K^{(m)} * G^{m,b}_\epsilon g\;.$$
Since $\xi_\epsilon$ is smooth, the model $\Pi^{(m)}_\epsilon$ is made of smooth functions and~\cite[Remark 3.15]{Hairer2014} shows that
$$ \cR_\epsilon(\cG^{m,b}_\epsilon g \cdot \Xi)(x) = (\Pi^{(m)}_\epsilon)_x (\cG^{m,b}_\epsilon g \cdot \Xi)(x) = G^{m,b}_\epsilon g(x)(\xi_\epsilon+C_\epsilon^{(m)})\;.$$
This yields the identity
$$ (-\Delta + m + b + \xi_\epsilon + C_\epsilon^{(m)}) G^{m,b}_\epsilon g = g\;,$$
in the sense of distributions. Furthermore $G^{m,b}_\epsilon g$ vanishes at the boundary (resp. is periodic under periodic b.c.). Therefore, for any $g,\bar{g} \in L^2$, a simple integration by parts shows that we have
\begin{align*}
\langle G^{m,b}_\epsilon g, \bar g\rangle &= \langle G^{m,b}_\epsilon g , (-\Delta + m + b + \xi_\epsilon + C_\epsilon^{(m)}) G^{m,b}_\epsilon \bar{g} \rangle_{L^2}\\
&= \langle (-\Delta + m + b + \xi_\epsilon + C_\epsilon^{(m)}) G^{m,b}_\epsilon g, G^{m,b}_\epsilon \bar{g} \rangle_{L^2}\\
&= \langle g , G^{m,b}_\epsilon \bar{g}\rangle_{L^2}\;.
\end{align*}
We turn to the resolvent identity. Since $\xi_\epsilon$ is a continuous function, it is standard to show that $-\Delta + \xi_\epsilon$ is a self-adjoint operator with domain $H^2 \cap H^1_0$ (resp. $H^2 \cap H^1_{\mbox{\tiny per}}$ in the periodic case), bounded from below and with pure point spectrum. The classical resolvent identity then yields \eqref{Eq:ResolvEps}.
\end{proof}
We now have all the ingredients at hand to prove Proposition \ref{Prop:Compact}.
\begin{proof}[Proof of Proposition \ref{Prop:Compact}]
Given Lemma \ref{Lemma:Compact}, it only remains to show that the definition of $G^a$ does not depend on the choice of the pair $(m,b)$ and that the resolvent identity holds.\\
Let $m,m'$ be such that $\P(E_m\cap E_{m'}) > 0$. From the convergence of the models stated in Proposition \ref{Prop:Renorm}, there exists a subsequence $\epsilon_k$ such that $(\Pi^{(m)}_{\epsilon_k},\Gamma^{(m)}_{\epsilon_k})$ and $(\Pi^{(m')}_{\epsilon_k},\Gamma^{(m')}_{\epsilon_k})$ converge almost surely. We now work on the event $E_m \cap E_{m,\epsilon_k} \cap E_{m'} \cap E_{m',\epsilon_k}$. Combining the resolvent identity of Lemma \ref{Lemma:ResolvId} with the continuity bound \eqref{Eq:ContBnd}, we deduce by passing to the limit on $\epsilon_k\downarrow 0$ that for $\P$-almost all $\omega \in E_m\cap E_{m'}$ we have for all $b,b'\in (-2,2)$
$$ G^{m,b} - G^{m',b'}= (m'+b'-m-b + C^{(m') - (1)} - C^{(m) - (1)}) G^{m',b'} G^{m,b}\;.$$
Since there are countably many events $E_m$, we deduce that up to a $\P$-null set the definition of $G^a, a \in \ccA$ is unambiguous and that the resolvent identity holds.
\end{proof}

\subsection{Definition of the operator}

For any given $a\in \ccA$, we set
$$ \cH := (G^{a})^{-1} - a\;.$$
This is an unbounded operator on $L^2$ whose domain is the range of $G^a$: from the resolvent identity, this range (and actually the definition of the operator $\cH$) does not depend on $a$. The self-adjointness of $G^{a}$ ensures that $\cH$ is also self-adjoint.

\begin{remark}
The domain of $\cH$ consists of all $\cR f$ where $f$ is obtained through the fixed point argument of Proposition \ref{Prop:Resolv} for some $g\in L^2$. Note that these functions can be decomposed as the sum of $(-\Delta + a)^{-1} g$, which lies in $H^2$, and $\cR \tilde{f}$, where $\tilde{f}$ is a modelled distribution in $\cD^{\tilde{\gamma},\sigma}_{p,\infty}$ with $\tilde{\gamma} < 4-d/2$.
\end{remark}

The spectral theorem for compact operators, see for instance~\cite[Th VI.16]{ReedSimon}, yields the existence of a complete orthonormal basis $(\varphi_n)_{n\ge 1}$ of eigenfunctions with associated eigenvalues $(\mu_n^{(a)})_{n\ge 1}$:
$$ G^a \varphi_n = \mu_n^{(a)} \varphi_n\;,\quad n\ge 1\;.$$
We deduce that $\cH$ admits a pure point spectrum given by the set $\{(\mu_n^{(a)})^{-1} - a:n\ge 1\}$, and that the eigenfunctions are the $\{\varphi_n:n\ge 1\}$.\\

For $a\in\ccA$, we do not know whether $G^a$ is a positive operator, and the ordering of the eigenvalues $(\mu_n^{(a)})_{n\ge 1}$ is arbitrary. To deal with positive operators, we introduce the following subset of $\ccA$
\begin{align*}
\ccA_+ := \Big\{ m+b+C^{(m) - (1)}: m\ge 1, b\in (-2,2) &\mbox{ s.t. }\forall m'\ge m\\
\$ \Pi^{(m')} \$ (1+ \$ \Gamma^{(m')} \$) < \frac12 A(m') &\quad\mbox{ and }\quad |C^{(m'+1) - (1)} -C^{(m') - (1)}| < 1\Big\}\;.
\end{align*}
\begin{lemma}
Almost surely, $\ccA_+$ belongs to the resolvent set of $\cH$ and there exists $a_+\in\R$ such that $\ccA_+ = (a_+,+\infty)$.
\end{lemma}
\begin{proof}
The Borel-Cantelli Lemma combined with \eqref{Eq:BoundModel} and \eqref{Eq:BoundCaa'} ensures $\ccA_+$ is almost surely non-empty. Since $C^{(m)-(1)} = o(m)$, we easily deduce that $\inf_{m\ge 1} m+C^{(m)-(1)}>-\infty$ and $\sup_{m\ge 1} m+C^{(m)-(1)}=+\infty$, so that $\ccA_+$ is bounded from the left and unbounded from the right. If $|C^{(m'+1) - (1)} -C^{(m') - (1)}| < 1$ then one can find $b$ and $b'$ such that $m'+b+C^{(m') - (1)} = m'+1+b'+C^{(m'+1) - (1)}$. This ensures that $\ccA_+$ is connected, and the statement follows.
\end{proof}
\noindent Consequently, for any $a\in\ccA_+$ the operator $\cH + a$ is positive and therefore $G^a$ is positive as well, so we can assume that the sequence $(\mu_n^{(a)})_{n\ge 1}$ is non-increasing and converges to $0$. We thus set $\lambda_n := (\mu_n^{(a)})^{-1} - a$ and we obtain
$$ \cH \varphi_n = \lambda_n \varphi_n\;,\quad n\ge 1\;.$$
The sequence $(\lambda_n)_{n\ge 1}$ is non-decreasing. Moreover, we have the bound
\begin{equation}\label{Eq:TrivBound}
\lambda_n \ge - \min(\ccA_+)\;,\quad \forall n\ge 1\;.
\end{equation}

\medskip

Note that this construction also applies to the operators $G^a_\epsilon$ and allows to define $\cH_\epsilon$. We have
$$ \cH_\epsilon = -\Delta + \xi_\epsilon + C^{(1)}_\epsilon\;.$$

Before we proceed to the proof of Theorem \ref{Th:Main}, we state a general result on compact, self-adjoint operators.
\begin{lemma}\label{Lemma:CVCompact}
Let $G_\epsilon$ be a sequence of non-negative self-adjoint operators on $L^2$ that converges for the operator norm to some positive self-adjoint operator $G$. Assume that there exists a compact set $K \subset L^2$ such that for all $\epsilon > 0$, $G_\epsilon B(0,1) \subset K$. Then, $(\mu_{n,\epsilon},\varphi_{n,\epsilon})_{n\ge 1}$ converges to $(\mu_n,\varphi_n)_{n\ge 1}$ in the sense of Remark \ref{Remark:CV}.
\end{lemma}
\begin{proof}
We first show that for any given $n\ge 1$, we have
\begin{equation}\label{Eq:DiffEigenvalues}
| \mu_{n,\epsilon} - \mu_n | \le \sup_{g \in L^2:\|g\|_{L^2} = 1} \| (G_\epsilon - G)g \|_{L^2}\;.
\end{equation}
From the min-max formula for non-negative, compact self-adjoint operators, we have
$$\mu_{n,\epsilon} = \sup_{F\subset L^2: \mbox{\scriptsize dim} (F) = n} \inf_{f\in F: \| f\|_{L^2} = 1} \langle G_\epsilon f , f \rangle_{L^2}\;,$$
(here $F$ ranges over all linear subspaces of $L^2$) and similarly for $\mu_n$ and $G$. Without loss of generality, assume that $\mu_{n,\epsilon} \le \mu_n$. Taking $F$ as the linear span of $\varphi_1,\ldots,\varphi_n$ we have
$$ \mu_n = \inf_{f\in F:\|f\|_{L^2}=1} \langle G f,f\rangle_{L^2}\;,\qquad \mu_{n,\epsilon} \ge \inf_{f\in F:\|f\|_{L^2}=1} \langle G_\epsilon f,f\rangle_{L^2}\;.$$
Furthermore, we have
$$ \inf_{f\in F:\|f\|_{L^2}=1} \langle G_\epsilon f,f\rangle_{L^2} \ge \inf_{f\in F:\|f\|_{L^2}=1} \langle G f,f\rangle_{L^2} + \inf_{f\in F:\|f\|_{L^2}=1} \langle (G_\epsilon-G)f,f\rangle_{L^2}\;,$$
so
$$ 0 \le \mu_n - \mu_{n,\epsilon}\le \sup_{f\in F:\|f\|_{L^2}=1} \langle (G - G_\epsilon)f,f\rangle_{L^2} \le \sup_{g \in L^2:\|g\|_{L^2} = 1} \| (G_\epsilon - G)g \|_{L^2}\;,$$
and \eqref{Eq:DiffEigenvalues} follows. We thus deduce the convergence of the eigenvalues.\\
We pass to the convergence of the eigenfunctions. Fix $n\ge 1$. The collection of vectors $(\varphi_{1,\epsilon},\ldots,\varphi_{n,\epsilon})_{\epsilon >0}$ takes values in $K\times\ldots\times K$ which is a compact subset of $L^2\times\ldots\times L^2$. Let $(\psi_1,\ldots,\psi_n)$ be the limit of a converging subsequence (for simplicity we keep the notation $\epsilon$ for the subsequence). Note that orthonormality is preserved under the limit so that $\langle \psi_k , \psi_\ell\rangle_{L^2} = \delta_{k,\ell}$. Taking the limit as $\epsilon\downarrow 0$ of
$$ \mu_{k,\epsilon} \varphi_{k,\epsilon} = G_\epsilon \psi_k + G_\epsilon (\varphi_{k,\epsilon} - \psi_k)\;,$$
we deduce that $\mu_k \psi_k = G \psi_k$. Consequently $(\psi_1,\ldots,\psi_n)$ are linearly independent eigenfunctions associated to the sequence $\mu_1,\ldots,\mu_n$. This holds true for any converging subsequence, we thus deduce the convergence of the statement.
\end{proof}

\begin{proof}[Proof of Theorem \ref{Th:Main}]
In this proof, we consider a parameter $a=a(m)$ of the form $a=m + C^{(m) - (1)}$ for some $m\in\N$. Since $\ccA_+$ is a.s. unbounded from the right, we have that $\P(a \in \ccA_+) \to 1$ as $m\uparrow \infty$. In the sequel, we take $m$ large enough such that $\P(a \in \ccA_+) > 0$. By \eqref{Eq:BoundCaa'} and Proposition \ref{Prop:Renorm}, the following holds true. For $\epsilon$ small enough there exists $b_\epsilon \in (-2,2)$ such that $m+b_\epsilon + C^{(m)}_\epsilon - C^{(1)}_\epsilon = a$ and $\P(a \in \ccA_\epsilon | a \in \ccA_+) \to 1$ as $\epsilon \downarrow 0$. Note that $b_\epsilon \downarrow 0$ as $\epsilon\downarrow 0$.\\
Using the continuity bound \eqref{Eq:ContBnd}, we deduce that conditionally given $a\in \ccA_+$,
$$ \sup_{g \in L^2:\|g\|_{L^2} = 1} \| (G^a_\epsilon\un_{a\in\ccA_\epsilon} - G^a)g \|_{L^2}\;,$$
converges in probability to $0$. Here $G^a_\epsilon\un_{a\in\ccA_\epsilon}$ denotes the operator that equals $G^a_\epsilon$ when $a\in\ccA_\epsilon$ and that is null otherwise. Furthermore, there exists $\delta >0$ and a deterministic constant $C_{a,L} >0$ such that for every $\epsilon > 0$ we have
$$ \sup_{g\in L^2: \|g\|_{L^2}\le 1} \| G^a_\epsilon\un_{a\in\ccA_\epsilon} g\|_{\cB^{\delta}_{2,2}} < C_{a,L}\;.$$
Recall that the embedding of $\cB^{\delta}_{2,2}$ into $L^2$ is compact. Hence, conditionally given $a\in \ccA_+$, the sequence of operators $G^a_\epsilon \un_{a\in\ccA_\epsilon}$ converges in probability to $G^a$, and maps the unit ball of $L^2$ into the centred ball of radius $C_{a,L}$ of $\cB^{\delta}_{2,2}$, which is a compact subset of $L^2$.\\
Let $\Sp(G^a_\epsilon\un_{a\in\ccA_\epsilon})$ denotes the spectrum of $G^a_\epsilon\un_{a\in\ccA_\epsilon}$. We aim at showing that, conditionally given $a\in\ccA_+$, the probability that this spectrum is contained in $(0,\infty)$ goes to $1$. First, since $G^a$ is positive on the event $a\in\ccA_+$, the convergence of $G^a_\epsilon \un_{a\in\ccA_\epsilon}$ towards $G^a$ ensures that for all $r>0$ we have
$$ \P\big(\Sp(G^a_\epsilon\un_{a\in\ccA_\epsilon}) \cap (-\infty,-r] \ne \emptyset \,|\, a\in\ccA_+\big) \to 0\;,\quad \epsilon\downarrow 0\;.$$
Second, combining Lemma \ref{Lemma:Moments}, Lemma \ref{Lemma:ModelLattice} and \eqref{Eq:BoundCaa'}, we deduce that
$$ \sup_{\epsilon \in (0,1)} \P(\min \ccA_{\epsilon,+} > t) \to 0\;,\quad t\to\infty\;.$$
Consequently applying \eqref{Eq:TrivBound} to the operator $\cH_\epsilon$ we get
$$ \sup_{\epsilon \in (0,1)}\P(\lambda_{1,\epsilon} < -t) \to 0\;,\quad t\to\infty\;,$$
so that
$$ \varlimsup_{r\downarrow 0} \sup_{\epsilon\in (0,1)} \P\big(\Sp(G^a_\epsilon\un_{a\in\ccA_\epsilon}) \cap (-r,0) \ne \emptyset \,|\, a\in\ccA_+\big) = 0\;.$$
Therefore, if we let $\cS_\epsilon := \{\Sp(G^a_\epsilon\un_{a\in\ccA_\epsilon}) \subset (0,\infty)\}$ then we have shown that
$$ \P\big(\cS_\epsilon \,|\, a\in\ccA_+\big) \to 1\;,\quad \epsilon \downarrow 0\;.$$
This implies that on the event $a\in\ccA_+$, $G^a_\epsilon\un_{a\in\ccA_\epsilon}\un_{\cS_\epsilon}$ converges in probability to $G^a$. Lemma \ref{Lemma:CVCompact} then ensures that $(\mu_{n,\epsilon}^{(a)},\varphi_{n,\epsilon})_{n\ge 1}$ converges in probability to $(\mu_{n}^{(a)},\varphi_{n})_{n\ge 1}$ in the sense of Remark \ref{Remark:CV}. Since
$$ \lambda_{n,\epsilon} = (\mu_{n,\epsilon}^{(a)})^{-1} - a\;,$$
we deduce that on the event $a\in\ccA_+$, $(\lambda_{n,\epsilon})_{n\ge 1}$ converges in probability to $(\lambda_n)_{n\ge 1}$, thus concluding the proof.
\end{proof}

\subsection{Tail estimate}

The goal of this subsection is to prove the tail estimate stated in Theorem \ref{Th:Tail}. To keep notations simple, we let the domain be $(-1,1)^d$ but the arguments apply \textit{mutatis mutandis} to a domain of arbitrary given size. We follow the strategy of proof of~\cite{AllezChouk} that relies on an identity in law between the spectrum of the Anderson hamiltonian on $(-1,1)^d$ and the spectrum of some rescaled version of this operator on a larger domain.

We let $\xi$ be a white noise on $\R^d$ (in the periodic case: $\xi$ is periodic with domain $(-1,1)^d$). For any given $L>0$, we set
$$ \tilde{\xi}_\epsilon(x) = L^{-2} \xi_\epsilon(x/L)\;,\quad x\in \R^d\;,$$
and we let $\tilde{\xi}$ be its limit as $\epsilon\downarrow 0$. One can check that $\tilde{\xi}_\epsilon$ and $\tilde{\xi}$ have the same laws as $L^{d/2-2}\zeta_{\epsilon L}$ and $L^{d/2-2}\zeta$, where $\zeta$ is a white noise on $\R^d$ (in the periodic case: $\zeta$ is periodic with domain $(-L,L)^d$)  and $\zeta_{\epsilon L}=\zeta*\varrho_{\epsilon L}$.\\

We define the model $\tilde{\Pi}_\epsilon^{(a)}$ by renormalising the canonical model based on $\tilde{\xi}_\epsilon$, we proceed in the same way as in Subsection \ref{Sec:Models}, the only difference lies in the values of the constants. For $d=2$ we take $\tilde{c}_\epsilon^{(a)} = L^{-2} c_{\epsilon L}^{(a)}$, while for $d=3$ we take
\begin{equation}\label{Eq:Renormtilde}
\tilde{c}_\epsilon^{(a)} = L^{-1} c_{\epsilon L}^{(a)}\;,\quad \tilde{c}^{(a),1,1}_\epsilon = L^{-2} c^{(a),1,1}_{\epsilon L}\;,\quad \tilde{c}^{(a),1,2}_\epsilon = L^{-2}c^{(a),1,2}_{\epsilon L}\;.
\end{equation}
In this context, the sequence of models $(\tilde{\Pi}_\epsilon^{(a)},\tilde{\Gamma}_\epsilon^{(a)})$ converges in probability to a limiting model $(\tilde{\Pi}^{(a)},\tilde{\Gamma}^{(a)})$. An elementary computation shows that for any $a\ge 1$, $L^{-2} C_\epsilon^{(a)} - \tilde{C}_\epsilon^{(a)}$ converges as $\epsilon\downarrow 0$ to a finite limit denoted $\tilde{\delta}_L^{(a)}$, and that $\tilde{\delta}_L^{(a)}$ vanishes when $L\to\infty$.\\

We then denote by $(\tilde{\lambda}_n,\tilde{\varphi}_n)_{n\ge 1}$ the sequence of eigenvalues/eigenfunctions associated to the Anderson hamiltonian on $(-L,L)^d$ driven by the rescaled white noise $\tilde{\xi}$, and by $({\lambda}_n,{\varphi}_n)_{n\ge 1}$ the same quantities but for the Anderson hamiltonian on $(-1,1)^d$ driven by the white noise $\xi$. The following observation is the cornerstone of the proof of the tail estimate, and was originally proven by Allez and Chouk~\cite{AllezChouk} in dimension $2$ for the first eigenvalue.

\begin{lemma}\label{Lemma:IdLaw}
We have the following equality in law
$$ \big(L^{-2} \lambda_n \big)_{n\ge 1} = \big(\tilde{\lambda}_n + \tilde{\delta}_L^{(1)}\big)_{n\ge 1}\;.$$
\end{lemma}
\begin{proof}
The key observation is that for any $\epsilon \in (0,1)$, $(-L,L)^d \ni x\mapsto \varphi_{n,\epsilon}(x/L)$ is an eigenfunction of the operator
$$ -\Delta + \tilde{\xi}_\epsilon + \tilde{C}_\epsilon^{(1)}\;,\quad x\in (-L,L)^d\;,$$
with eigenvalue $L^{-2}(\lambda_{n,\epsilon} + L^2 \tilde{C}_\epsilon^{(1)} - C_\epsilon^{(1)})$. Necessarily the latter quantity coincides with $\tilde{\lambda}_{n,\epsilon}$ so that we get
$$ L^{-2} \lambda_{n,\epsilon} = \tilde{\lambda}_{n,\epsilon} + L^{-2} C_\epsilon^{(1)} - \tilde{C}_\epsilon^{(1)}\;.$$
Recall that $L^{-2} C_\epsilon^{(1)} - \tilde{C}_\epsilon^{(1)}$ converges to a finite limit $\tilde{\delta}_L^{(1)}$ as $\epsilon\downarrow 0$. On the other hand, Theorem \ref{Th:Main} shows that $\lambda_{n,\epsilon}$ converges in law to $\lambda_{n}$, and the very same arguments show that $\tilde{\lambda}_{n,\epsilon} \to \tilde{\lambda}_{n}$, thus concluding the proof.
\end{proof}

Fix some constant $c>0$ (that will be adjusted later). We deduce from the lemma that we have
$$ \P(\lambda_n \le - c L^2) = \P(\tilde{\lambda}_n + \tilde{\delta}_L^{(1)} \le - c)\;.$$
Since $\tilde{\delta}_L^{(1)} \rightarrow 0$ as $L\to\infty$, to establish Theorem \ref{Th:Tail} it suffices to bound from above $\P(\tilde{\lambda}_n \le - c/2)$ and from below $\P(\tilde{\lambda}_n \le - 2 c)$ uniformly over all $L$ large enough. These two bounds will be obtained separately.

To prove the upper bound, we need the following bound on the norm of the models, whose proof is postponed to Subsection \ref{Subsec:Techos}.

\begin{proposition}\label{Prop:BoundModel}
There exist $K,C>0$ such that we have
$$ \P\big(\sup_{m\ge 1} \$ \tilde{\Pi}^{(m)} \$ (1+ \$ \tilde{\Gamma}^{(m)} \$) > K\big) \lesssim e^{-C L^{4-d}}\;,$$
uniformly over all $L\ge 1$.
\end{proposition}

Let us now define the constant $c$. Let $A^{-1}$ be the reciprocal of the function $a\mapsto A(a)$ appearing in Proposition \ref{Prop:Resolv}. We take
$$ c = 6 \max(1 , A^{-1}(2K))\;,$$
where $K$ is the constant appearing in Proposition \ref{Prop:BoundModel}.

\begin{proof}[Proof of Theorem \ref{Th:Tail} - upper bound]
From \eqref{Eq:Renormtilde} and using a similar computation as for \eqref{Eq:BoundCaa'}, we deduce that for $L$ large enough we have
$$ \sup_{m\ge 1} |\tilde{C}^{(m+1) - (1)} - \tilde{C}^{(m) - (1)}| < 1\;.$$
On the event
$$ \bigcap_{m\ge 1} \Big\{\$\tilde{\Pi}^{(m)}\$ (1 + \$\tilde{\Gamma}^{(m)}\$) \le K \Big\}\;,$$
and for all $L$ large enough, the set $\tilde{\ccA}_+$ contains $\{m+\tilde{C}^{(m)-(1)}:m > A^{-1}(2K)\}$. Consequently \eqref{Eq:TrivBound} yields
$$ \tilde{\lambda}_n \ge -(m_0 + \tilde{C}^{(m_0) - (1)})\;,\quad\forall n\ge 1\;,$$
where $m_0 := \lceil A^{-1}(2K) \rceil$. Since $\tilde{C}^{(m_0) - (1)} \to 0$ as $L\to \infty$, we get
$$ \tilde{\lambda}_n \ge - \frac32 m_0 > -\frac{c}{2}\;,$$
for all $L$ large enough. Henceforth, for all $L$ large enough
$$ \P(\tilde{\lambda}_n \le - c/2) \le \P\Big(\sup_{m\ge 1} \$\tilde{\Pi}^{(m)}\$ (1 + \$\tilde{\Gamma}^{(m)}\$) > K \Big)\;.$$
By Proposition \ref{Prop:BoundModel}, we thus get the existence of a constant $b>0$ such that
$$ \P(\lambda_n \le - x)  \le e^{-b x^{2-\frac{d}{2}}}\;,$$
uniformly over all $x>0$ large enough.
\end{proof}
To bound from below $\P(\tilde{\lambda}_n \le - 2 c)$, the idea is to introduce a deterministic, smooth function $h$ in such a way that the $n$-th eigenvalue of the Anderson hamiltonian on $(-L,L)^d$ with potential $h$ lies below $-2c$ (actually, below $-3c$ to have some wiggle room). To that end, it suffices to take $h$ as the sum of $n$ appropriately chosen disjoint bumps. Then, by the Cameron-Martin Theorem one can estimate the probability that the noise $\tilde{\xi}$ is close to $h$ so that the continuity bound \eqref{Eq:ContBnd} allows to compare the eigenvalues of the two corresponding operators.
\begin{proof}[Proof of Theorem \ref{Th:Tail} - lower bound]
Let $f_1$ be a smooth function, supported in $B(0,1/2)$ such that $\|f_1\|_{L^2} = 1$. Set
$$ b := -3c - \| \nabla f_1\|_{L^2}^2 < 0\;,$$
and let $\chi_1$ be a non-positive smooth function, supported in $B(0,1)$, that equals $b$ on $B(0,1/2)$ and that is larger than $b$ elsewhere. Then, for every $k\in \{2,\ldots,n\}$ we let $(\chi_k,f_k)$ be translates of $(\chi_1,f_1)$ in such a way that for any $k \ne \ell$ the supports of $\chi_k$ and $f_\ell$ do not intersect. We set
$$ h := \sum_{k=1}^n \chi_k\;.$$
Consider the Anderson hamiltonian associated to the noise $h$ on $(-L,L)^d$, with $L$ large enough for the supports of all the previous functions to fall within $(-L,L)^d$. Denote by $(\bar{\lambda}_j)_{j\ge 1}$ its eigenvalues and let $(\bar{\Pi}^{(a)},\bar{\Gamma}^{(a)})$ be the corresponding model (no renormalisation is required since the noise is smooth). Note that for every $k\in\{1,\ldots,n\}$ we have
$$\| \nabla f_k \|_{L^2}^2 + \int h f_k^2 = \| \nabla f_k \|_{L^2}^2 + b \int f_k^2 = -3c\;.$$
The functions $f_1,\ldots, f_n$ are linearly independent since their supports are disjoint, and therefore the min-max formula yields the following upper bound
$$\bar\lambda_n \le -3c\;.$$
On the other hand, since the $\|h\|_\infty \le -b$, the min-max formula yields  the following lower bound
$$ b \le \bar\lambda_n\;.$$

Since $h$ is a given smooth function, the norm of the models $(\bar\Pi^{(m)},\bar\Gamma^{(m)})$ is uniformly bounded over all $m$ and all $L$. We introduce the event
$$ E_{m} := \big\{ m + \tilde{C}^{(m) - (1)} \in \tilde{\ccA}_+ \;,\quad m \in \bar{\ccA}_+ \big\}\;.$$
By Proposition \ref{Prop:BoundModel}, there exists $m_0 \ge 1$ such that
$$ \P(E_{m_0}) \to 1\;,\quad L\to\infty\;.$$
Set $\tilde{a} = m_0 + \tilde{C}^{(m_0) - (1)}$ and $\bar a = m_0$, and note that $|\tilde{a}-\bar a| \lesssim L^{-1}$. By the min-max formula (as in \eqref{Eq:DiffEigenvalues}), on the event ${E}_{m_0}$ we have
$$ \big|\tilde{\mu}_n^{(\tilde{a})} - \bar{\mu}_n^{(\bar a)}\big| \le \sup_{g \in L^2:\|g\|_{L^2} = 1} \| (\tilde{G}^{(\tilde{a})} - \bar{G}^{(\bar a)})g \|_{L^2}\;.$$
Recall that in dimension $d=3$ we have to embed $L^p$ into $L^2$ and that the norm of the embedding grows with $L$ like $L^{18 \kappa}$, see Remark \ref{Remark:Embedding}. Using the continuity bound \eqref{Eq:ContBnd}, we deduce that
$$ \big|\tilde{\mu}_n^{(\tilde{a})} - \bar{\mu}_n^{(\bar a)}\big| \lesssim L^{18 \kappa} (1+\$\bar{Z}^{(m_0)}\$) \$ \tilde{Z}^{(m_0)} ; \bar{Z}^{(m_0)} \$\;,$$
uniformly over all $L$ large enough.\\
On the other hand, we have on the event ${E}_{m_0}$
$$ \tilde{\mu}_n^{(\tilde{a})} - \bar{\mu}_n^{(\bar a)} = \frac1{\tilde{\lambda}_n + m_0 + \tilde{C}^{(m_0) - (1)}} - \frac1{\bar{\lambda}_n+m_0}\;.$$
Using Lemma \ref{Lemma:ModelLattice}, we thus deduce that there exists $\delta >0$ such that for all $L$ large enough we have
$$ \P(\tilde{\lambda}_n \le - 2c) \ge \P(\$ \tilde{\Pi}^{(m_0)} - \bar{\Pi}^{(m_0)} \Latt < \delta L^{-18 \kappa})\;.$$

\medskip

Observe that for every $\tau \in \cT_{<\gamma}$, $\tilde{\Pi}^{(m_0)}_x \tau$ lives in some inhomogeneous Wiener chaos associated to the Gaussian noise $\tilde{\xi}$. To estimate the probability on the right, we thus shift the noise by $h$ and bound the difference between the models based on $\tilde{\xi}+h$ and on $h$. More precisely, the Cameron-Martin Theorem ensures that
\begin{align*}
\P(\$ \tilde{\Pi}^{(m_0)} - \bar{\Pi}^{(m_0)} \$_\Lambda < \delta L^{-18 \kappa}) &= \E\Big[e^{-\langle \tilde{\xi},h\rangle - \frac12 L^{4-d} \|h\|_{L^2}^2} \,\un_{\{\$ \tilde{\Pi}^{(m_0)}(\tilde{\xi}+h) - \bar{\Pi}^{(m_0)} \$_\Lambda < \delta L^{-18 \kappa}\}}\Big]\\
&\ge e^{-\frac12 L^{4-d} \|h\|_{L^2}^2} \Big(\P\big(\$ \tilde{\Pi}^{(m_0)}(\tilde{\xi}+h) - \bar{\Pi}^{(m_0)} \$_\Lambda < \delta L^{-18 \kappa}\big) - \frac12\Big)\;,
\end{align*}
where we used that $\langle \tilde{\xi},h\rangle$ is a centred Gaussian r.v.~so that its probability of being positive equals $1/2$. Note that $\|h\|_{L^2}^2$ does not depend on $L$. If
\begin{equation}\label{Eq:ModelShift}
\P\big(\$ \tilde{\Pi}^{(n_0)}(\tilde{\xi}+h) - \bar{\Pi}^{(n_0)} \$_\Lambda < \delta L^{-18 \kappa}\big) \to 1\;,\quad L\to\infty\;,
\end{equation}
then we obtain the desired lower bound. The proof of \eqref{Eq:ModelShift} is postponed to the next subsection.
\end{proof}

\subsection{Some bounds on the models}\label{Subsec:Techos}

We start with a bound on the growth in $L$ of the exponential moments of the norm of the model (here the model is based on a white noise and is taken on $(-L,L)^d$). Recall the two sets $\cU$ and $\cF$ defined in Subsection \ref{Subsection:RS}. Denote by $\cW$ the set of symbols that lie in $\cU$ or $\cF$ that are not monomials and whose homogeneity is below $\gamma$. Denote by $\|\tau\|$ the number of formal occurrences of $\Xi$ in $\tau \in \cW$.
\begin{lemma}\label{Lemma:Moments}
There exist two constants $\lambda > 0$ and $\nu >0$ such that
$$ \sup_{a\ge 1} \sup_{L\ge 1} \frac1{L^d} \sum_{\tau \in \cW} \E\Big[\exp\Big(\lambda a^\nu \$ \Pi^{(a)} \tau \Latt^{\frac2{\|\tau\|}}\Big)\Big] < \infty\;,$$
and
$$ \sup_{\epsilon \in (0,1)}\sup_{a\ge 1} \sup_{L\ge 1} \frac1{L^d} \sum_{\tau \in \cW} \E\Big[\exp\Big(\lambda a^\nu \$ \Pi^{(a)}_\epsilon \tau \Latt^{\frac2{\|\tau\|}}\Big)\Big] < \infty\;.$$
\end{lemma}
Note that the exponent $2/{\|\tau\|}$ is natural: roughly speaking, the white noise has gaussian tails so that symbols that contain $k$ instances of the white noise should have tails that decay like $e^{-x^{2/k}}$.
\begin{proof}
Recall from~\cite[Section 1.4.3]{Nualart} that for all r.v.~$X$ in the $k$-th inhomogeneous Wiener chaos associated to the stochastic $L^2$ space generated by the white noise, and for any $p > 2$:
$$ \E[|X|^{p}] \le C_{k,p} \E[|X|^2]^{p/2}\;,$$
where $C_{k,p} = (p-1)^{pk/2}$.\\
Fix $\tau$ and let $\varphi$ be the scaling function of some compactly supported wavelet basis of regularity $r > -|\Xi|$. By the construction of the renormalised model~\cite[Section 10.2]{Hairer2014}, there exist $\delta > 0$ and $C>0$ such that we have
$$ \E\big[\langle (\Pi^{(a)}_\epsilon)_x \tau , \varphi^n_x \rangle^2\big] \le C 2^{-n(d+2|\tau| + \delta)}\;,\qquad \E\big[\langle \Pi^{(a)}_x \tau , \varphi^n_x \rangle^2\big] \le C 2^{-n(d+2|\tau| + \delta)}\;,$$
for all $x\in \Lambda_n \cap (-L,L)^d$, all $n\ge n_a$, all $a\ge 1$ and all $\epsilon \in (0,1)$. Note that $C$ does not depend on $a\ge 1$ since the bounds \eqref{Eq:BndP} on the kernels hold uniformly over all $a\ge 1$.\\
The remainder of the proof is identical for the models $\Pi^{(a)}_\epsilon$ and $\Pi^{(a)}$, so we only consider the latter. We have
\begin{align*}
\E\Big[{\$ {\Pi}^{(a)} \tau \Latt}^{\frac{2p}{\|\tau\|}}\Big] &= \E\Big[\sup_{n\ge n_a} \sup_{x\in \Lambda_n\cap (-L,L)^d} \Big(\frac{\langle \Pi^{(a)}_x \tau , \varphi^n_x\rangle}{2^{-n(\frac{d}{2} + |\tau|)}}\Big)^{\frac{2p}{\|\tau\|}}\Big]\\
&\lesssim \sum_{n\ge n_a} 2^{nd} L^d C_{\|\tau\|,\frac{2p}{\|\tau\|}} \E\Big[ \Big(\frac{\langle \Pi^{(a)}_0 \tau , \varphi^n_0\rangle}{2^{-n(\frac{d}{2} + |\tau|)}}\Big)^{2}\Big]^{\frac{p}{\|\tau\|}}\\
&\lesssim L^d C_{\|\tau\|,\frac{2p}{\|\tau\|}} \sum_{n\ge n_a} 2^{nd} 2^{-n\delta \frac{p}{\|\tau\|}}\\
&\lesssim 2^{-n_a(\delta \frac{p}{\|\tau\|}-d)} L^d C_{\|\tau\|,\frac{2p}{\|\tau\|}}\;,
\end{align*}
as soon as $p > d \|\tau\| / \delta$. This allows to bound the $p$-th moment of $\$ {\Pi}^{(a)} \tau \Latt^{\frac{2}{\|\tau\|}}$ for all $p\ge p_0$ where $p_0 :=  2d \|\tau\| / \delta$. For $p\le p_0$, we write
\begin{align*}
\E\Big[\$ {\Pi}^{(a)} \tau \Latt^{\frac{2p}{\|\tau\|}}\Big] &\le \E\Big[\$ {\Pi}^{(a)} \tau \$_{|\tau|}^{\frac{2p_0}{\|\tau\|}}\Big]^{\frac{p}{p_0}}\le \Big(2^{-n_a\delta \frac{p_0}{2\|\tau\|}} L^d C_{\|\tau\|,\frac{2p_0}{\|\tau\|}}\Big)^{\frac{p}{p_0}}\lesssim 2^{-n_a\delta \frac{p}{2\|\tau\|}} L^d C_{\|\tau\|,\frac{2p_0}{\|\tau\|}}\;.
\end{align*}
This being given, we write
\begin{align*}
\E\Big[\exp(\lambda a^\nu \$ {\Pi}^{(a)} \tau \Latt^{\frac{2}{\|\tau\|}})\Big] &= \sum_{p\ge 0} \frac{\lambda^p a^{\nu p}}{p!} \E\Big[\$ {\Pi}^{(a)} \tau \$_{\Lambda}^{\frac{2p}{\|\tau\|}}\Big]\\
&\lesssim \sum_{p\le p_0} \frac{\lambda^p a^{\nu p}}{p!}2^{-n_a\delta \frac{p}{2\|\tau\|}} L^d C_{\|\tau\|,\frac{2p_0}{\|\tau\|}} + \sum_{p>p_0} \frac{\lambda^p a^{\nu p}}{p!} 2^{-n_a\delta \frac{p}{2\|\tau\|}}L^d C_{\|\tau\|,\frac{2p}{\|\tau\|}}\;.
\end{align*}
By choosing $\nu$ smaller than $\delta/(4\|\tau\|)$, we can bound $a^{\nu p}2^{-n_a\delta \frac{p}{2\|\tau\|}}$ by $1$. The first sum is then bounded by a term of order $L^d$. Given the expression of the constant $C_{k,p}$ recalled at the beginning of the proof, we deduce that for $\lambda$ small enough, the second sum is bounded by a term of order $L^d$ as well. This concludes the proof of the lemma.
\end{proof}
\begin{proof}[Proof of Proposition \ref{Prop:BoundModel}]
By Lemma \ref{Lemma:ModelLattice}, it suffices to prove that for every $\tau \in \cW$ there exist $K_0,C>0$ such that uniformly over all $L\ge 1$
$$\P(\sup_{m\ge 1} \$ \tilde{\Pi}^{(m)} \tau \Latt > K_0) \lesssim e^{-C L^{4-d}}\;.$$
Since $\tilde{\xi}$ has the same law as $L^{\frac{d}{2} - 2}\xi$, we deduce the following identity in law
$$ \tilde{\Pi}^{(m)}_x \tau = L^{(\frac{d}{2} - 2) \|\tau\|} \Pi^{(m)}_x \tau \;,$$
where on the r.h.s.~the model is built on $(-L,L)^d$ with a white noise. Using Lemma \ref{Lemma:Moments}, we obtain
\begin{align*}
\P(\$ \tilde{\Pi}^{(m)} \tau \Latt > K_0) &= \P( \$ {\Pi}^{(m)} \tau \Latt > K_0 L^{(2-\frac{d}{2}) \|\tau\|})\\
&\le e^{-\lambda m^\nu K_0^{\frac{2}{\|\tau\|}} L^{4-d}} \E\Big[\exp(\lambda m^\nu \$ {\Pi}^{(m)} \tau \$_{\Lambda}^{\frac{2}{\|\tau\|}})\Big]\\
&\le L^d e^{-\lambda m^\nu K_0^{\frac{2}{\|\tau\|}} L^{4-d}} \sup_{L\ge 1} \frac1{L^d} \E\big[\exp(\lambda m^\nu \$ {\Pi}^{(m)} \tau \$_{\Lambda}^{\frac{2}{\|\tau\|}})\big]\\
&\lesssim L^d e^{-\lambda m^\nu K_0^{\frac{2}{\|\tau\|}} L^{4-d}}\;,
\end{align*}
Therefore, there exists $C>0$ such that
\begin{align*}
\P(\sup_{m\ge 1}\$ \tilde{\Pi}^{(n)} \tau \$_{|\tau|} > K_0) \lesssim \sum_{m\ge 1} L^d e^{-\lambda m^\nu K_0^{\frac{2}{\|\tau\|}} L^{4-d}}\lesssim e^{-C L^{4-d}}\;,
\end{align*}
and this completes the proof.
\end{proof}
We now proceed to the proof of \eqref{Eq:ModelShift}. We will rely on graphical notations introduced in~\cite[Section 5]{HaiPar} and on a bound on generalised convolutions established in~\cite[Appendix A]{HairerQuastel}: we will not recall the whole set of notations and definitions, and refer the reader to the aforementioned references.
\begin{proof}[Proof of \eqref{Eq:ModelShift}]
For simplicity, we drop the superscript $m_0$ and we write $\tilde{\Pi}^h$ for the model associated to the noise shifted in direction $h$. By the arguments presented in this subsection, it suffices to show that for all $\tau \in \cW$ with $|\tau| < 0$, we have
\begin{equation}\label{Eq:ShiftBound}
\E\Big[\langle \big(\tilde{\Pi}^h - \bar{\Pi}\big)_x \tau , \varphi^\lambda_x\rangle^2\Big] \lesssim L^{d - 4} \lambda^{2|\tau|}\;,
\end{equation}
uniformly over all $\varphi\in\ccB^r$, all $x\in (-L,L)^d$ and all $\lambda \in (0,2^{-m_0}]$.\\
We start with $\tau = \Xi$, in which case we have
$$ \tilde{\Pi}^h_x \Xi = \tilde{\xi} + h\;,$$
so that
$$ \big(\tilde{\Pi}^h - \bar{\Pi}\big)_x \Xi = \tilde{\xi}\;.$$
Since $\tilde{\xi}$ equals in law $L^{\frac{d}{2}-2} \zeta$ where $\zeta$ is a white noise, we deduce that \eqref{Eq:ShiftBound} holds in this case.\\
Next, we consider $\tau = \Xi\, \cI(\Xi)$ (in dimension $d\ge 2$ only). To facilitate the analysis, we rely on the graphical notations introduced in~\cite[Section 5]{HaiPar}. In particular a circle denotes an instance of $\Xi$ (recall that here $\Xi$ corresponds to $\tilde{\xi}$ which is a scaled white noise), a black dot an integration variable, a black arrow the kernel $P^{(a)}_+$ and a red arrow the test function. We introduce an additional graphical notation: a square will denote an instance of $h$. Note that since $h$ is smooth, each square has the very same ``behaviour'' as a black dot. We then have
\begin{equation}
\bigl(\tilde{\Pi}^h\, \<Xi2>\bigr)(\varphi^\lambda)  =\;
\begin{tikzpicture}[scale=0.35,baseline=0.3cm]
	\node at (0,-1)  [root] (root) {};
	\node at (-2,1)  [var] (left) {};
	\node at (0,3)  [var] (left1) {};
	
	\draw[testfcn] (left) to  (root);	
	\draw[kernel1] (left1) to (left);
\end{tikzpicture}\;
+
\begin{tikzpicture}[scale=0.35,baseline=0.3cm]
	\node at (0,-1)  [root] (root) {};
	\node at (-2,1)  [varH] (left) {};
	\node at (0,3)  [var] (left1) {};
	
	\draw[testfcn] (left) to  (root);	
	\draw[kernel1] (left1) to (left);
\end{tikzpicture}\;
+
\begin{tikzpicture}[scale=0.35,baseline=0.3cm]
	\node at (0,-1)  [root] (root) {};
	\node at (-2,1)  [var] (left) {};
	\node at (0,3)  [varH] (left1) {};
	
	\draw[testfcn] (left) to  (root);	
	\draw[kernel1] (left1) to (left);xx
\end{tikzpicture}\;
+
\begin{tikzpicture}[scale=0.35,baseline=0.3cm]
	\node at (0,-1)  [root] (root) {};
	\node at (-2,1)  [varH] (left) {};
	\node at (0,3)  [varH] (left1) {};
	
	\draw[testfcn] (left) to  (root);	
	\draw[kernel1] (left1) to (left);
\end{tikzpicture}\;
 - L^{d-4}\;
\begin{tikzpicture}[scale=0.35,baseline=0.3cm]
	\node at (0,-1)  [root] (root) {};
	\node at (0,3)  [dot] (top) {};
	
	\draw[testfcn,bend right = 60] (top) to  (root);
	
	\draw[kernel,bend left = 60] (top) to (root);
\end{tikzpicture}
\end{equation}
The fourth term cancels out with $\bigl(\bar{\Pi} \<Xi2>\bigr)(\varphi^\lambda)$. The fifth term is deterministic and goes to $0$ as $L\to\infty$. To bound the first three terms, we rely on a bound on generalised convolutions established by Hairer and Quastel in~\cite[Appendix A]{HairerQuastel}. In a nutshell, to every kernel in the graphs one associates a pair of labels $(a,r)$: $a$ stands for the singularity of the kernel and $r$ to the order of renormalisation. Then, under the four conditions 1. ,2. , 3. and 4. of~\cite[Assumption A.1]{HairerQuastel} one concludes that the integral associated to the square of the $L^2$ norm of the random variable encoded by the graph decays at some explicit speed in $\lambda$ see~\cite[Theorem A.3]{HairerQuastel}. In our setting, we want to check that we get the right exponent in $\lambda$ and a prefactor in $L$ that decays fast enough. Note that the prefactor in $L$ will come from the fact that our noise is scaled by $L^{d/2-2}$ so that each circle in the graph yields a prefactor $L^{d-4}$ in the square of the $L^2$-norm.\\

The first term falls in the scope of Theorem A.1 as it is shown in~\cite[Section 5.2.1]{HaiPar}. Since in addition the noise is scaled by $L^{\frac{d}{2}-2}$ we get a multiplicative prefactor $L^{2d-8}$ in the bound, which is sufficient for our purpose. The second and third terms should be seen as slight modifications of the first term, but unfortunately do not match condition 4. of Theorem A.1. Indeed, if we focus on the second term the corresponding labelled graph is given by
$$ 
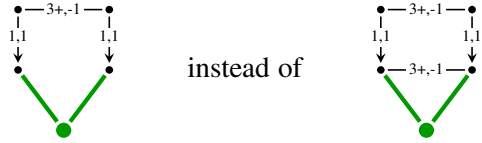
\begin{tikzpicture}[scale=0.4,baseline=0.3cm]
	\node at (0,-1)  [root] (root) {};
	\node at (-1.5,1)  [dot] (left) {};
	\node at (-1.5,3)  [dot] (left1) {};
	\node at (1.5,1) [dot] (variable1) {};
	\node at (1.5,3) [dot] (variable2) {};
	
	\draw[dist] (left) to (root);
	\draw[dist] (variable1) to (root);
	
	\draw[->,generic] (left1) to  node[labl,pos=0.45] {\tiny 1,1} (left);
	\draw[->,generic] (variable2) to node[labl,pos=0.45] {\tiny 1,1} (variable1);
	\draw[generic] (variable2) to node[labl] {\tiny 3+,-1} (left1); 
\end{tikzpicture}
\qquad\mbox{ instead of }\qquad
 \begin{tikzpicture}[scale=0.4,baseline=0.3cm]
	\node at (0,-1)  [root] (root) {};
	\node at (-1.5,1)  [dot] (left) {};
	\node at (-1.5,3)  [dot] (left1) {};
	\node at (1.5,1) [dot] (variable1) {};
	\node at (1.5,3) [dot] (variable2) {};
	
	\draw[dist] (left) to (root);
	\draw[dist] (variable1) to (root);
	
	\draw[->,generic] (left1) to  node[labl,pos=0.45] {\tiny 1,1} (left);
	\draw[->,generic] (variable2) to node[labl,pos=0.45] {\tiny 1,1} (variable1);
	\draw[generic] (variable2) to node[labl] {\tiny 3+,-1} (left1); 
	\draw[generic] (variable1) to  node[labl] {\tiny 3+,-1} (left); 
\end{tikzpicture}
$$
While the graph has the same vertices, an edge labelled $(3+, -1)$ has been removed. A careful inspection of conditions 1., 2. and 3. of Assumption A.1 shows that removing such an edge preserves the required inequalities. On the other hand, removing such an edge prevents condition 4. from being satisfied.\\

However, one can modify the statement of the theorem in order not to assume condition 4. Set
$$ \beta := \min_{\bar{\bbV} \subset \bbV\backslash \bbV_*} \bigg( \sum_{e\in \hat\bbE(\bar\bbV)\setminus \hat\bbE^{\downarrow}_+(\bar \bbV) }  \hat a_e 
 +\sum_{e\in \hat \bbE^{\uparrow}_+(\bar\bbV)}  r_e
- \sum_{e \in \hat \bbE^\downarrow_+(\bar \bbV)} (r_e-1)
- |\bar \bbV| d \bigg)\;.$$
Condition 4. requires $\beta > 0$. As observed in~\cite[Remark A.12]{HairerQuastel}, if $\beta < 0$ then Theorem A.3 still holds upon replacing $\tilde\alpha$ by $\tilde\alpha+\beta$.\\
Since we removed an edge labelled $(3+,-1)$ and since originally $\beta$ was strictly positive, we deduce that for the second term above we have $\beta > -3$. On the other hand, $\tilde{\alpha}$ increases by $3$ upon removing an edge $(3+,-1)$ so that we still have a bound of the desired order in $\lambda$. Furthermore, each occurrence of $\Xi$ in the graph produces a prefactor $L^{d-4}$ in the bound of the square of the $L^2$-norm and this completes the proof of \eqref{Eq:ShiftBound} for $\tau = \Xi\cI(\Xi)$.\\
The proof of \eqref{Eq:ShiftBound} for more complicated trees proceeds from exactly the same arguments: one observes a cancellation between the two models, and all remaining trees can be bounded using Theorem A.3 or its modification presented above.
\end{proof}

\bibliographystyle{Martin}
\bibliography{library}

\end{document}